\documentclass[11pt,a4paper]{amsart}
\usepackage{amsmath,amssymb,amsthm, amsfonts,enumerate,mathtools,pinlabel,float,stmaryrd}
\usepackage[mathscr]{eucal}
\usepackage{amsbsy}
\usepackage{bm}
\usepackage{tikz}
\usetikzlibrary{arrows}
\usetikzlibrary{decorations.markings}
\usepackage{graphicx}

\newcommand{\cev}[1]{\reflectbox{\ensuremath{\vec{\reflectbox{\ensuremath{#1}}}}}}

\newtheorem{theorem*}{Theorem}

\newtheorem{theorem}{Theorem}[section]
\newtheorem{lemma}[theorem]{Lemma}
\newtheorem{proposition}[theorem]{Proposition}

\theoremstyle{definition}

\newtheorem{definition}[theorem]{Definition}
\newtheorem{example}[theorem]{Example}
\newtheorem{notation}[theorem]{Notation}

\newtheorem{remark}[theorem]{Remark}

\numberwithin{equation}{section}
\theoremstyle{plain}

\newtheorem{corollary}[theorem]{Corollary}

\numberwithin{equation}{section}

\newcommand{\homeo}{\text{Homeo}}
\newcommand{\Mod}{\text{Mod}}
\newcommand{\G}{\mathcal{G}}

\newcommand{\p}{\mathcal{P}}
\newcommand{\N}{\mathcal{N}}
\newcommand{\B}{\mathcal{B}}
\newcommand{\C}{\mathcal{C}}
\renewcommand{\S}{\mathbb{S}}

\renewcommand{\l}{\operatorname{{\llbracket}}}
\renewcommand{\r}{\operatorname{{\rrbracket}}}
\newcommand{\lp}{\operatorname{{\llparenthesis}}}
\newcommand{\rp}{\operatorname{{\rrparenthesis}}}
\begin{document}

\title[Geometric realizations of cyclic actions]{Geometric realizations of \\cyclic actions on surfaces}

\author{Shiv Parsad}
\address{Department of Mathematics\\
Indian Institute of Science Education and Research Bhopal\\
Bhopal Bypass Road, Bhauri \\
Bhopal 462 066, Madhya Pradesh\\
India}
\email{shivparsad@iiserb.ac.in}

\author{Kashyap Rajeevsarathy}
\address{Department of Mathematics\\
Indian Institute of Science Education and Research Bhopal\\
Bhopal Bypass Road, Bhauri \\
Bhopal 462 066, Madhya Pradesh\\
India}
\email{kashyap@iiserb.ac.in}
\urladdr{https://home.iiserb.ac.in/$_{\widetilde{\phantom{n}}}$kashyap/}

\author{Bidyut Sanki}
\address{Department of Mathematics\\
Institute of Mathematical Sciences\\
IV Cross Road, CIT Campus, Taramani\\
Chennai 600 113 Tamil Nadu\\ 
India}
\email{bidyut.iitk7@gmail.com}



\maketitle

\begin{abstract}
 Let  $ \Mod(S_g)$ denote the mapping class group of the closed orientable surface $S_g$ of genus $g\geq 2$, and let $f\in\Mod(S_g)$ be of finite order. We give an inductive procedure to construct an explicit hyperbolic structure on $S_g$ that realizes $f$ as an isometry. In other words, this procedure yields an explicit solution to the Nielsen realization problem for cyclic subgroups of $\Mod(S_g)$. Furthermore, we give a purely combinatorial perspective by showing how certain finite order mapping classes can be viewed as fat graph automorphisms. As an application of our realizations, we determine the sizes of maximal reduction systems for certain finite order mapping classes. Moreover, we describe a method to compute the image of finite order mapping classes and the roots of Dehn twists, under the symplectic representation $\Psi: \text{Mod}(S_g) \to \text{Sp}(2g; \mathbb{Z})$. 
\end{abstract}

\section{Introduction}
\label{intro}
Let $S_g$ be a closed orientable surface of genus $g$, and let $\Mod(S_g)$ denote the mapping class group of $S_g$. The Nielsen realization problem~\cite{JN3} asks whether a given finite subgroup $H$ of $\Mod(S_g)$, for $g \geq 2$, can be realized as a group of isometries of some hyperbolic metric on $S_g$. Stated in other words, the question asks whether the action of $H$ on the Teichm\"{u}ller space has a fixed point. While Nielsen~\cite{JN1} solved this problem for cyclic subgroups, it was later settled for arbitrary subgroups by Kerkchoff~\cite{SK1}. Taking the existential nature of these results into consideration, a natural question is whether one can determine an explicit hyperbolic structure on $S_g$ that realizes $H$ as a group of isometries. In this paper, we answer this question for the case when $H$ is cyclic. 

For $g \geq 1$, let $H = \langle h \rangle$ be a cyclic subgroup of $\Mod(S_g)$ of order $n$ that acts on $S_g$ with a quotient orbifold~\cite[Chapter 16]{WT} $\mathcal{O}_h := S_g/H$ of genus $g_0(h)$. If there exists an imbedding of $\sigma : S_g \to \mathbb{R}^3$ under which $h$ can realized as the restriction of a rotation of $\mathbb{R}^3$, then we say that $h$ is a \textit{rotational action}.  Such an imbedding $\sigma$ induces a Riemannian metric $\rho_\sigma$ on $S_g$, whose conformal class contains a hyperbolic metric $\mu_\sigma$ (which is unique up to scaling) that realizes this rotation as an isometry. As obtaining an explicit description of such a metric is rather straightforward, we shall  focus our attention on other types of cyclic actions. Our realizations of non-rotational cyclic actions fall into one general philisophy, which involves the decompostion of the given action into its irreducible components. To this effect, we appeal to a result of Gilman~\cite{JG3} which states that an action $h$ is irreducible if, and only if $\mathcal{O}_h$ is (topologically) a sphere with three cone points. Moreover, a beautiful result of Kulkarni~\cite{RK} shows the existence of a polygon (with a suitable side-pairing) whose rotation realizes a given $C_n$-action $h$ with a common fixed point. Taking inspiration from these results, we introduce the notion of a \textit{Type 1} cyclic action $h$, for which $\mathcal{O}_h$ has three cone points, at least one of which is a distinguished cone point of order $n$. Denoting this distinguished cone point by $P$, we see that $P$ lifts under $h$ to a fixed point $P_h \in S_g$, around which $h$ induces a local rotation of $\theta_h = 2\pi c^{-1}/n$, where $\gcd(c,n)=1$. For $i = 1,2$, let $P_i \in \mathcal{O}_h$ be the remaining two cone points of orders $n_i$ that lift to orbits $O_i$ of sizes $n/n_i$ with local rotation angles $2\pi c_i^{-1}/n_i$, where $\gcd(c_i,n_i) = 1$. The following result gives a geometric realization of an arbitrary Type 1 action.
\begin{theorem*}\label{res:1}
For $g \geq 2$, a Type 1 action $h$ on $S_g$ can be realized explicitly as the rotation $\theta_h$ of a hyperbolic polygon $\p_h$ with a suitable side-pairing $W(\p_h)$, where $\p_h$ is a hyperbolic  $k(h)$-gon with
$$ \small k(h) := \begin{cases}
2n(1+2g_0(h)), & \text { if } n_1,n_2 \neq 2, \text{ and } \\
n(1+4g_0(h)), & \text{otherwise, }
\end{cases}$$
and for $0 \leq m\leq n-1$, 
$$ \small
W(\p_h) =
\begin{cases}
\displaystyle  
  \prod_{i=1}^{n} Q_ia_{2i-1} a_{2i} \text{ with } a_{2m+1}^{-1}\sim a_{2z}, & \text{if } k(h) = 2n, \text{ and } \\
\displaystyle
 \prod_{i=1}^{n} Q_ia_{i} \text{ with } a_{m+1}^{-1}\sim a_{z}, & \text{otherwise,}
\end{cases}$$
where $\displaystyle z \equiv m+qj \pmod{n}, \,q= (n/n_2)c^{-1},j=n_{2}-c_{2}$, and \\ $\small Q_r = \prod_{s=1}^{g_0(h)} [x_{r,s},y_{r,s}], \, 1 \leq r \leq n.$
\end{theorem*}
\noindent  

A non-rotational action $h$ that is not of Type 1, is called a \textit{Type 2} action. Our understanding of the realizations of arbitrary Type 2 actions can be motivated with the following subtle, yet intriguing, geometric construction. For $i=1,2$, consider Type 1 actions $h_i$ on surfaces $S_{g_i}$ inducing local rotational angles $\theta_i$ around points in distinguished orbits $O_i \subset S_{g_i}$ of size $k$, so that $\theta_1 + \theta_2 \equiv 0 \pmod{2\pi}$. We call such a pair of orbits as \textit{compatible}. For $1 \leq j \leq k$, we remove pairwise disjoint invariant disks $D_{i,j} \subset S_{g_i}$ around the points in $O_i$ ensuring that the $k$ pairs of disks $D_{1,j},D_{2,j}$ have the same hyperbolic area $a$ (under the structures $\p_{h_i}$), and then identify the resultant boundary components $\partial D_{1,j}$ with $\partial D_{2,j}$ that have the same length $\ell = \ell(a)$. This defines an explicit hyperbolic structure that realizes an action $h$ on $S_{g'}$, where $g' = g_1+g_2 + k-1$. We say the Type 2 action $h$ is realizable as a \textit{compatible pair $(h_1,h_2)$} of Type 1 actions. The hyperbolic structure that realizes $h$ depends on the structures $\p_{h_i}$, and the parameter $\ell$. Alternatively, one can also realize cyclic actions by identifying pairs of boundary components obtained by deleting a pairwise disjoint  collection of invariant disks around compatible orbits induced within the same surface. We call such actions as \textit{self compatible}, and the structures that realize these actions are analogous to those that are realized as compatible pairs. Using the result of Gilman~\cite{JG3} mentioned earlier, we can infer that every Type 2 action $h$ has a maximal reductive system $\C$. Consequently, $h$ will induce a finite order map $\tilde{h}$ on the (possibly disconnected) surface obtained by capping $\overline{S_g \setminus \C}$. By analyzing the surface orbits of $\tilde{h}$, we can obtain a decomposition of $h$ into finitely many compatibilities (of the types mentioned above) between irreducible cyclic actions. 

\begin{theorem*}\label{res:2}
For $g \geq 2$, any Type 2 action $h$ on $S_g$ has a decomposition into finitely many compatibilities between irreducible cyclic actions. Consequently, the hyperbolic structures that realize $h$ as isometries can be built from the structures that realize the compatible pairs and self-compatibilities in the decomposition.
\end{theorem*}

\noindent Thus, Theorems~\ref{res:1} and~\ref{res:2} together give an explicit solution to the Nielsen realization problem for all non-rotational cyclic actions. 

In Section~\ref{sec:3}, we give a combinatorial perspective to these realizations by establishing a correspondence between certain finite order mapping classes and \textit{fat graph automorphisms}~\cite{BS1}. Let $\Gamma$ be fat graph with one boundary component with a vertex set $V$ and an edge set $E$. An automorphism $F$ of $\Gamma$ of order $n$ is said to be \textit{irreducible} if with either $|E| = n$ and $|V/\langle F \rangle| = 1$ or $|E| = 2n$ and $|V / \langle F \rangle | = 2$. In this direction, we have the following result. 
\begin{theorem*}
There is a bijective correspondence between irreducible Type 1 actions and irreducible fat graph automorphisms. 
\end{theorem*}

\noindent Furthermore, we show that a compatibility between irreducible Type 1 actions determines a notion of compatibility between the corresponding irreducible fat graph automorphisms, and this leads to the following result. 

\begin{theorem*}
 Every cyclic action that decomposes into compatibilities between finitely many irreducible Type 1 actions can be viewed as a fat graph automorphism. 
 \end{theorem*}

In Section~\ref{sec:3}, we give several applications of our geometric realizations. As a consequence of Theorem~\ref{res:2}, we determine the size of maximal reduction systems for certain finite order maps, as detailed in the following result. 
\begin{theorem*}
Let $h$ be a cyclic action on $S_g$.
\begin{enumerate}[(i)]
\item If $h$ is of Type 1, then there exists a maximal reduction system $\C$ for $h$ such that
$$|\C|  = 
\begin{cases}
n(3g_0(h)-1), & \text{if } g_0(h)>1, \text{ and} \\
2n, & \text{if } g_0(h)=1.
\end{cases}$$
\item If $h$ is a Type 2 action realizable as a compatible pair $(h_1,h_2)$ of Type 1 actions $h_i$ on $S_{g_i}$, then there exists a maximal reduction system $\C$ for $h$ such that
$$\small |\C| = \begin{cases}
n(3g_0(h_1)+3g_0(h_2)-2)+k, & \text { if }g_0(h_1),g_0(h_2)>1 ,  \\
n(3g_0(h_1)-1)+k+2n, &\text{ if }g_0(h_1)>1,g_0(h_2)=1,\\
n(3g_0(h_2)-1)+k+2n, &\text{ if }g_0(h_2)>1,g_0(h_1)=1,\\
n(3g_0(h_2)-1)+k, &\text{ if }g_0(h_2)>1,g_0(h_1)=0,\\
n(3g_0(h_1)-1)+k, &\text{ if }g_0(h_1)>1,g_0(h_2)=0,\text{ and}\\
k, & \text{otherwise,}
\end{cases}$$
where $k = g -g_1-g_2+1$.
\end{enumerate}

\end{theorem*} 

\noindent We now turn our attention to some algebraic consequences of our realizations. Let $\Psi: \text{Mod}(S_g) \to \text{Sp}(2g; \mathbb{Z})$ be the symplectic representation of $\Mod(S_g)$. As a consequence of these realizations, in Section~\ref{sec:2}, we derive an explicit procedure for determining the $\Psi(h)$. It is not hard to see that for a rotational action $h$ that is either free or of order $n \geq 3$, $\Psi(h)$ is a simple permutation matrix. However, when $h$ is a non-free action of order 2, it has $2k$ fixed points, and $\Psi(h)$ has the form 
 $\begin{tiny} 
 \begin{bmatrix}  
 -I_{2k-2}  & 0 \\
 0  & P
 \end{bmatrix}
 \end{tiny},$ where $I_{2k-2}$ is the identity matrix of order $2k-2$, and $P$ is some permutation matrix of order $2(g-k)+2$. When $h$ is a Type 1 action, we give an explicit handle normalization algorithm to obtain a symplectic basis in terms of the letters in $W(\p_h)$ (see Proposition~\ref{prop:normal}). This enables us to compute $\Psi(h)$ by reading the action from $\p_h$ (see Theorem~\ref{thm:rep type 2}). We then use Theorem~\ref{res:2} to obtain the symplectic representations of generic cyclic actions. As the restriction $\Psi$ to finite order mapping classes is faithful~\cite[Chapter 6]{FM}, one can perceive this as a significant step towards finding an explicit word representation for finite order elements of $\Mod(S_g)$ in terms of its standard Dehn twist generators. 

Let $C$ be a nonseperating simple closed curve in $S_g$, and let $t_C$ denote the left-handed Dehn twist about $C$. It is known~\cite{MS,km} that a root $\phi$ of $t_C$ of degree $n$ corresponds to a finite order action $h_{\phi}$ on $S_{g-1}$ that has a distinguished pair $P_i$ (for $i= 1,2$) of fixed points around which the induced angles $\theta_{P_i}$ satisfy $\theta_{P_1}+\theta_{P_2} \equiv 2\pi/n \pmod{2\pi}.$ The root $\phi$ is then realized by removing invariant disks around the $P_i$ and attaching an annulus $A_\phi$ with a $1/n^{th}$ twist across the resultant boundary components. So our procedure for computing the symplectic representations of finite order actions naturally extends to roots of $t_C$. In particular, we show that (see Theorem~\ref{thm:rep_roots}):
\begin{theorem*}
For $g \geq 2$, let $C$ be a nonseperating simple closed curve in $S_g$, and let $\phi \in \Mod(S_g)$ be a root of $t_C$ of degree $n$. Then 
$$\Psi(\phi) = \begin{bmatrix} 
                        \Psi(h_{\phi}) & B_1\\
                        B_2 & I_2
                       \end{bmatrix},$$
where $\Psi(h_\phi)$ is obtained using Theorem~\ref{res:2}, and the block $B_1$ (and hence $B_2$) is completely determined by the action of $h_{\phi}$ on the pair of standard homology generators determined by the attachment of $A_{\phi}$. 
\end{theorem*}
\noindent Suppose that $C$ is a separating curve in $S_g$ so that $S_g = S_{g_1} \#_C S_{g_2}$. Then a root $\phi$ of $t_C$ of degree $n$ corresponds to a pair $(h_\phi^1, h_\phi^2)$ of finite order maps, where for $i =1,2$, $h_{\phi}^i$ is an order $n_i$ map on $S_{g_i}$ with a distinguished fixed point $P_i$ satisfying $\theta_{P_1} + \theta_{P_2} \equiv 2\pi/n \pmod{n}$ with $n = \text{lcm}(n_1,n_2)$. An immediate consequence of this fact is that $\tiny \Psi(\phi) = \begin{bmatrix}
                        \Psi(h_\phi^1) & 0 \\
                        0 & \Psi(h_\phi^2)
                        \end{bmatrix},$
where the $\Psi(h_\phi^i)\text{'s}$ are obtained using Theorem~\ref{res:2}. Finally, we indicate how these results could be extended to compute the symplectic representations of roots of Dehn twists about multicurves~\cite{KP}. 

\section{Geometric realizations of cyclic actions}\label{sec:1}
\subsection{Preliminaries}\label{prel} 
Let $h:C_n \to \homeo(S_g)$ be a faithful $C_n$-action on $S_g$, where $C_n$ denotes the cyclic group of order $n$. We fix a generator $t$ for $C_ n$ and identify the finite order homeomorphism $h(t)\in \homeo(S_g)$ as the generating homeomorphism of the action. For convenience, we also use $h$ to denote the generating homeomorphism $h(t)$ of the action. 

In general, a $C_n$-action $h$ on $S_g$ induces a branched covering $$S_g \to \mathcal{O}_h := S_g/C_n,$$ which  have  $\ell$ branched points (or cone points) $x_1,\ldots ,x_{\ell}$ in the quotient orbifold $\mathcal{O}$ of orders $n_1, \ldots ,n_{\ell}$, respectively. For each $i$, the cone point $x_i$ lifts to an orbit of size $n/n_i$ on $S_g$, and the local rotation induced by $h$ around the points in the orbit is given by $2 \pi c_i^{-1}/n_i$, where $c_i c_i^{-1} \equiv 1 \pmod{n_i}$. This motivates the following definition. 

\begin{definition}\label{defn:data_set}
A \textit{data set of degree $n$} is a tuple
$$
D = (n,g_0, r; (c_1,n_1), (c_2,n_2),\ldots, (c_{\ell},n_{\ell})),
$$
where $n\geq 1$, $ g_0 \geq 0$, and $0 \leq r \leq n-1$ are integers, and each $c_i$ is a residue class modulo $n_i$ such that:
\begin{enumerate}[(i)]
\item $r > 0$ if, and only if $\ell = 0$, and when $r >0$, we have $\gcd(r,n) = 1$, 
\item each $n_i\mid n$,
\item for each $i$, $\gcd(c_i,n_i) = 1$, and
\item $\displaystyle \sum_{j=1}^{\ell} \frac{n}{n_j}c_j \equiv 0\pmod{n}$.
\end{enumerate}
The number $g$ determined by the equation
\begin{equation*}\label{eqn:riemann_hurwitz}
\frac{2-2g}{n} = 2-2g_0 + \sum_{j=1}^{\ell} \left(\frac{1}{n_j} - 1 \right) \tag{R-H}
\end{equation*}
is called the \emph{genus} of the data set, which we shall denote by $g(D)$. We denote the degree $n$ of the data set $D$ by $n(D)$, the number $r$ by $r(D)$, and the number $g_0$ by $g_0(D)$. 
\end{definition}

\noindent The following proposition (see \cite[Theorem 3.8]{KP} for a proof) establishes a correspondence between the conjugacy classes of cyclic actions and data sets.

\begin{proposition}\label{prop:ds-action}
Data sets of degree $n$ and genus $g$ correspond to conjugacy classes of $C_n$-actions on $S_g$. 
\end{proposition}

\noindent From here on, we will use data sets for describing the conjugacy classes of cyclic actions. Note that the quantity $r(D)$ associated with a data set $D$ will be non-zero if, and only if, $D$ represents a free rotation of $S_{g(D)}$ by $2\pi r(D)/n$. For brevity, we will avoid writing $r$ in the notation of a data set, whenever $r=0$. Equation~\ref{eqn:riemann_hurwitz} in Definition~\ref{defn:data_set} is the Riemann-Hurwitz equation associated with the branched covering $S_g \to \mathcal{O}_D (= S_g /C_n)$. Before diving into the geometric realizations of  cyclic actions, we classify $C_n$-actions on $S_g$ into three broad categories. 

\begin{definition}\label{def:types_of_actions}
Let $D$ be a $C_n$-action on $S_g$. Then $D$ is said to be a: 
\begin{enumerate}[(i)]
\item \textit{rotational action}, if either $r(D) \neq 0$, or $D$ is of the form 
$$(n,g_0;\underbrace{(s,n),(n-s,n),\ldots,(s,n),(n-s,n)}_{k \,pairs}),$$ 
for integers $k \geq 1$ and $0<s\leq n-1$ with $\gcd(s,n)= 1$, and $k=1$, if and only if $n>2$.  
\item \textit{Type 1 action}, if $\ell = 3$, and $n_i = n$ for some $i$. 
\item \textit{Type 2 action}, if $D$ is neither a rotational nor a Type 1 action. 
\end{enumerate}
\end{definition}  

\begin{remark}
In order to make sense of Definition~\ref{def:types_of_actions} (i), we first consider the case when a rotational action $D$ is a free rotation. In this case, the quotient map $S_g \to \mathcal{O}_D ( = S_g/C_n)$ is a covering space, and by the multiplicativity of the Euler characteristic, we have that the genus $g_0(D)$ of $\mathcal{O}_D$ satisfies $g_0(D) = 1+(g(D)-1)/n$. On the other hand, if $D$ is a non-free rotation, then $D$ will have $2k$ fixed points which are induced at the points of intersection of the axis of rotation with $S_g$. Moreover, these fixed points will form $k$ pairs of points $(x_i,x_i')$, for $1 \leq i \leq k$, such that the sum of the angles of rotation induced by $h$ around $x_i$ and  $x_i'$ add up to $0$ modulo $2\pi$. 
\end{remark}

As mentioned in Section~\ref{intro}, every rotational action $D$ on $S_g$, for $g \geq 2$, can be realized as a restriction $\mathcal{R}_D$ of a rotation of $\mathbb{R}^3$ under a fixed imbedding $\sigma:S_g \hookrightarrow \mathbb{R}^3$, the standard euclidean metric on $\mathbb{R}^3$ will induce a Riemannian metric $\rho_{\sigma}$ on $S_g$. By the uniformization theorem, the metric $\rho_{\sigma}$ is conformally equivalent to a unique hyperbolic metric $\mu_{\sigma}$ on $S_g$ that also realizes $\mathcal{R}_D$ as an isometry. For this reason, we will focus on non-rotational actions, beginning with the realizations of Type 1 actions. 

\subsection{Type 1 actions}
 In this subsection, we will show that a Type 1 action $D$ can be realized as the rotation of an even sided polygon with an appropriate side-pairing. We will assume throughout this section that $D$ has the form
 \[D=(n, g_0; (c_1 , n_1 ),(c_2, n_2), (c_ 3, n_ 3 )), \text{ where } n_3=n. \]
We motivate the realization of Type 1 actions with an example. 
 
\begin{example}
\label{eg:polygons}
Consider $C_5$-action
$$D=(5,0;(1,5),(3,5),(1,5))$$ on $S_2$.  This action can be realized in two possible ways, as shown in Figure~\ref{fig:2} below.
\begin{figure}[H]
\labellist
\small
\pinlabel $\huge \curvearrowleft$ at 195 185
\pinlabel $C$ at 195 150
\pinlabel $2\pi/5$ at 195 215
\pinlabel $A$ at 140 -10
\pinlabel $x_1$ at 204 -12
\pinlabel $B$ at 261 -10
\pinlabel $x_2$ at 317 23
\pinlabel $A$ at 360 67
\pinlabel $x_3$ at 393 117
\pinlabel $B$ at 400 177
\pinlabel $x_4$ at 390 232
\pinlabel $A$ at 360 287
\pinlabel $x_2$ at 317 318
\pinlabel $B$ at 260 352
\pinlabel $x_5$ at 200 356
\pinlabel $A$ at 140 352
\pinlabel $x_4$ at 85 325
\pinlabel $B$ at 40 287
\pinlabel $x_1$ at 10 232
\pinlabel $A$ at -5 177
\pinlabel $x_5$ at 5 117
\pinlabel $B$ at 30 67
\pinlabel $x_3$ at 75 23

\pinlabel $\huge \curvearrowleft$ at 655 185
\pinlabel $B$ at 655 150
\pinlabel $4\pi/5$ at 655 215
\pinlabel $C$ at 600 -10
\pinlabel $y_1$ at 664 -12
\pinlabel $A$ at 721 -10
\pinlabel $y_2$ at 777 23
\pinlabel $C$ at 821 67
\pinlabel $y_3$ at 853 117
\pinlabel $A$ at 860 179
\pinlabel $y_4$ at 850 232
\pinlabel $C$ at 826 287
\pinlabel $y_5$ at 777 318
\pinlabel $A$ at 724 354
\pinlabel $y_1$ at 660 356
\pinlabel $C$ at 600 352
\pinlabel $y_2$ at 545 325
\pinlabel $A$ at 500 287
\pinlabel $y_3$ at 470 232
\pinlabel $C$ at 455 177
\pinlabel $y_4$ at 465 117
\pinlabel $A$ at 490 67
\pinlabel $y_5$ at 535 23

\endlabellist
\centering
\includegraphics[width = 60 ex]{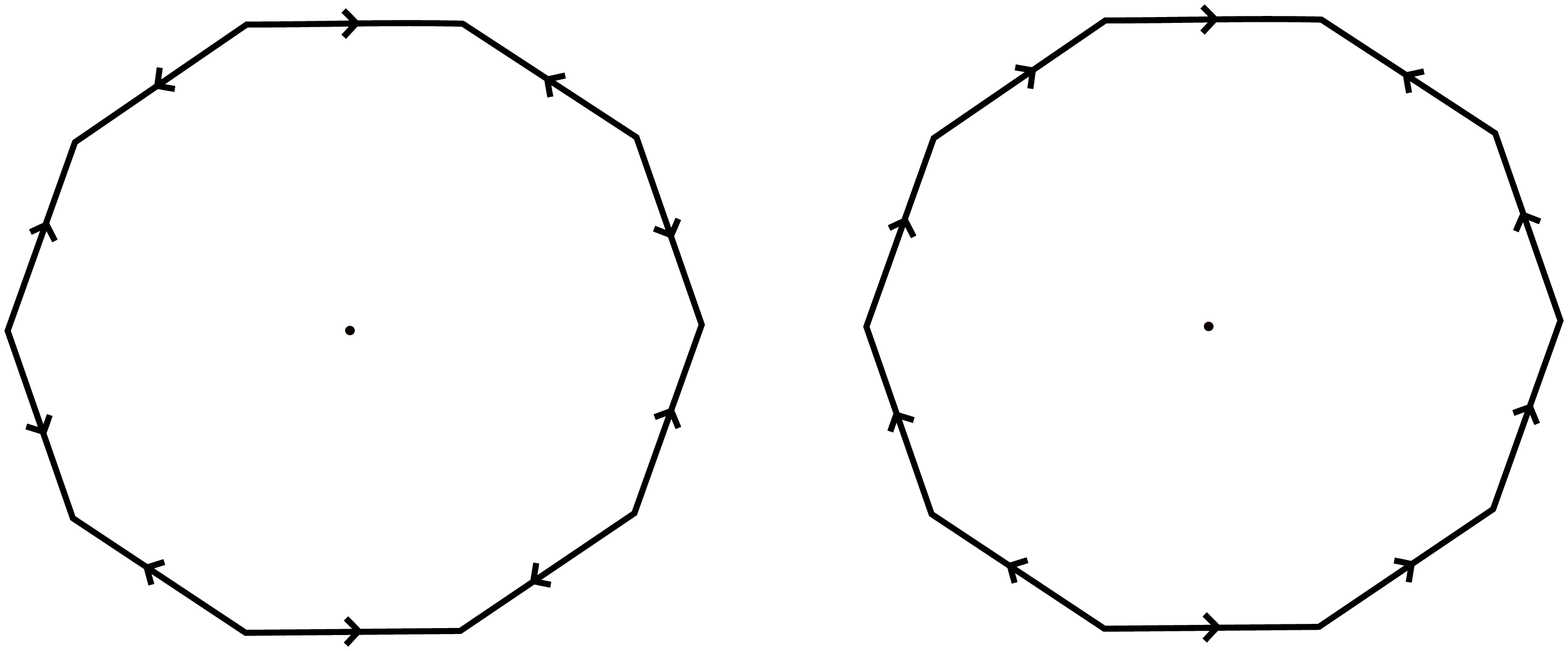}
\caption{Two realizations of the action $D$ on $S_2$.}
\label{fig:2}
\end{figure}
\end{example}
 
\noindent For a positive even integer $n$, we denote an $n$-gon by $\p_{n}$, the reduced boundary word of $\p_n$ (i.e. side-pairing) by $W(\p_n)$, and the quotient surface $W(\p_n)$ determines by $S(\p_n)$. It is implicitly assumed that $\p_n$ comes equipped with some $W(\p_n)$. Taking inspiration from Example~\ref{eg:polygons}, we have the following definition.

\begin{definition} \label{def:polygon_data set}
Let $D$ be an irreducible Type 1 action, that is, a Type 1 action with $g_0(D)=0$. Then:
\begin{enumerate}[(i)]
\item The \textit{topological polygon associated with} $D$ is defined by $\p_D := \p_{k(D)}$, where
$$k(D) := \begin{cases}
2n, & \text { if } n_1,n_2 \neq 2, \text{ and } \\
n, & \text{otherwise, }
\end{cases}$$
and for $0 \leq m \leq n-1$, we define
$$\small 
W(\p_{k(D)}) =
\begin{cases}
\displaystyle  
  \prod_{i=1}^{n} a_{2i-1} a_{2i} \text{ with } a_{2m+1}^{-1}\sim a_{2z}, & \text{if } k(D) = 2n, \text{ and } \\
\displaystyle
 \prod_{i=1}^{n} a_{i} \text{ with } a_{m+1}^{-1}\sim a_{z}, & \text{otherwise, }
\end{cases}$$
$$\text{where }\small z \equiv m+qj \pmod{n}, \,q= (n/n_2)c_3^{-1}, \text{ and } j=n_{2}-c_{2}.$$
\item The \textit{rotational angle associated with $D$} is defined by 
$$\theta_D =2\pi c_3^{-1}/n_3, \text{ where }c_3c_3^{-1} \equiv 1 \pmod{n_3}.$$
\end{enumerate}
\end{definition}

\noindent First, we will restrict our analysis to irreducible Type 1 actions on $S_g$. 
In this direction, our goal will be to establish the following theorem, which gives an explicit hyperbolic structure on $S_g$ that realizes such actions as isometries.
\begin{theorem} \label{main}
An irreducible Type 1 cyclic action $D$ can be realized as a rotation of $\p_D$ by $\theta_D$. 
\end{theorem}

\noindent In order to establish Theorem~\ref{main}, we will need the following technical lemma. 

\begin{lemma} \label{lg}
 Let $D$ be an irreducible Type 1 action. Then $S(\p_{k(D)}) \approx S_g$.  
\end{lemma}

\begin{proof}
First, we consider the case when $n_1,n_2 \neq 2$. For $0\leq i\leq 2n-1$, set $v_i,v_{i+1}$ to be the initial  and the terminal point of $a_i$, respectively. Let $X= \{v_0,~v_1,~\dots,~v_{2n-1} \}$, and let $G=\langle x \rangle$ be the cyclic group of order $2n$, generated by $x$. Consider the action of $G$ on $X$ given by $$x\cdot v_i = v_{i+1}, \text{ where } v_{2n}=v_0.$$ The side-pairing yields the relations $$v_{2i+1} \sim v_{2i+1+2s} \text{ and } v_{2i} \sim v_{2i+2s-2}, \text{ where } s = qj,$$ which can also be written as $$v_{2i+1} \sim x^{2s}\cdot v_{2i+1} \text{ and }
   v_{2i} \sim x^{2s-2}\cdot v_{2i}.$$ The subgroups $G_0=\langle x^{2s-2} \rangle,~G_1= \langle x^{2s} \rangle$ of $G$ act on sets $$X_0=\{v_0,~v_2,~\dots,~v_{2n-2}\} \text{ and } X_1=\{v_1,~v_3,~\dots,~v_{2n-1}\}$$ respectively. Observe that the canonical cellular decomposition of the resultant quotient surface has $|X_0/G_0|+|X_1/G_1|$ vertices (0-cells). Applying Burnside's Lemma~\cite[Theorem 3.22]{JR}, we have for $i = 0,1$, $$|X_{i} / G_{i}|=\frac{n}{|G_i|},$$ and so the number of vertices is given by $n/|G_0|+n/|G_1|$, where $|G_0|$ and $|G_1|$ are the orders of $2s-2$ and $2s$ in $\mathbb{Z}_{2n}$ respectively. 

 Since $\gcd(c_2,n_2)= \gcd(c_3,n)=1$, the order of $2s$ in $\mathbb{Z}_{2n}$ is $n_2$. Also, condition $(iii)$ on the data set $D$ implies that $$2s -2=-\frac{2n(n_1 -c_1)c^{-1}_{3}}{n_1} \pmod{2n},$$ and therefore the order of $2s-2$ in $\mathbb{Z}_{2n}$ is $n_1$. So, the number of vertices of the 
resultant surface is given by $\frac{n}{n_1}+\frac{n}{n_2}$. By the Riemann-Hurwitz equation, we have $$2-2g=-n+\frac{n}{n_1}+\frac{n}{n_2}+1,$$ and a simple Euler characteristic argument yields the result. A similar argument works for the case when either $n_1$ or $n_2$ equals $2$. 
\end{proof}

\noindent We are now ready to prove Theorem~\ref{main}.
\begin{proof}(Theorem~\ref{main}.)
Once again, we begin with the case when $n_1,n_2 \neq 2$. By Lemma~\ref{lg}, the pairing on the polygon $\p_{D}$ yields $S_g$. Define a homeomorphism $t: \p_{D} \to \p_{D}$ by $$t(a_{i})=a_{i+2c^{-1}_{3}}, \text{ for } 0\leq i\leq 2n-1,$$ that is, 
 a rotation of $\p_D$ by $\theta_D = 2\pi c_3^{-1}/n_3$. Clearly,  $t$ generates a $C_{n}$-action on $S_{g}$. We wish to prove that data set for $t$ is $D$.  

Under the branched covering $\pi: S_g\longrightarrow \mathcal{O} := S_{g}/C_n$, a cone point $p \in \mathcal{O}$ of order $k$ lifts to an orbit of size $n/k$. The center of $\p_{D}$ corresponds to a cone point of order $n$.  In the notation of Lemma~\ref{lg}, the sets $X_0\text{ and }X_1$ correspond to cone points of orders $n_1$ and $n_2$ respectively. 
 In the view of the condition $(iv)$ on the data set in Definition~\ref{defn:data_set}, it suffices to show that local rotational angle $(\theta_2)$ around the cone point of order $n_2$ is $2\pi c_2^{-1}/n_2$.
  
 Viewing the surface from $v_{2m+1}$, the angle between $a_{2m}$ and $a^{-1}_{2m+1}$ is $2\pi(n_{2}-1)/n_2$. The stabilizer of the orbit of size $n/n_2$ is generated by $t^{n/n_2}$. If $2\pi(n_{2}-1)/n_2=j^{\prime}\theta_{2}$, then 
 $$\displaystyle a_{2m+1}^{-1} \sim t^{\frac{n}{n_2}j^{\prime}}(a_{2m}) \sim a_z, \text{ where }  z \equiv 2m+\frac{2n}{n_2}{c_3}^{-1} j'\pmod{2n}.$$ By the side-pairing condition, 
 $$z \equiv 2m+j\frac{2n}{n_2}c^{-1}_{3} \pmod{2n},$$ which implies $$j^{\prime} \equiv -c_{2} \pmod{n_2}.$$ Therefore we have, $$-c_{2}\theta_{2} \equiv \frac{2\pi(n_{2}-1)}{n_2} \pmod{n_2},$$ that is, $\theta_{2}=2\pi c_2^{-1}/n_2$, which establishes the result. An analogous argument works for the other case.
\end{proof}

\begin{remark}
\label{rem:poly}
Suppose that $D$ is an irreducible Type 1 action with $g(D) \geq 2$. When one of the $n_i$ equals $2$, it follows from basic hyperbolic trigonometry that the polygon $\p_D$ determines a unique hyperbolic structure $\G(D)$ on $S_{g(D)}$ that realizes $D$ as an isometry. However, when $n_1,n_2 \neq 2$, it can be shown that there exists one-parameter family $\{\G(D,d)\}$ of hyperbolic structures carried by $\p_D$ in which each structure $\G(D,d)$ is uniquely determined by the hyperbolic distance $d$ between the center and one of the vertices of $\p_D$.
\end{remark}

\noindent The generalization of Theorem~\ref{main} to the case when $g_0(D)> 0$ will involve a  subtle topological construction, which we detail in the following remark. 

\begin{remark}\label{rem:l5_top}
Consider a $C_n$-action $D$ on $S_g$. Suppose that there exist distinct orbits $O_1$ and $O_2$ of $D$ of size $k$ such that the angles $\theta_i$ induced by $D$ around the points in the $O_i$ satisfy $$\theta_1 + \theta_2 \equiv 0 \pmod{2 \pi}.$$ 
For $1 \leq j \leq k$, we remove mutually disjoint disks $B_{i,j}$, cyclically permuted by $D$, around the points in the $O_i$, and attach $k$ annuli connecting the resultant boundary components $\partial B_{1,j}$ with the $\partial B_{2,j}$. This induces an action $\tilde{D}$ with $g(\tilde{D}) = g(D)+k$ and $g_0(\tilde{D}) = g_0(D)+1$. 
\end{remark}

\noindent We will now describe a necessary condition on a $C_n$-action $D$ that would facilitate a construction as in Remark~\ref{rem:l5_top}, assuming that the size of the chosen orbits $O_i$ is strictly less than $n$. 

\begin{definition}\label{def:self_comp_ds}
For $\ell \geq 4$, let
$$
D = (n,g_0; (c_1,n_1), (c_2,n_2),\ldots, (c_{\ell},n_{\ell})),
$$
be a $C_n$-action. Then $D$ is said to be $(r,s)$-\textit{self compatible}, if there exist $1 \leq r < s \leq \ell$ such that
\begin{enumerate}[(i)]
\item $n_r = n_s = m$, and
\item $\displaystyle c_r+c_s \equiv 0 \pmod{m}$. 
\end{enumerate}
\end{definition}

\noindent It is now apparent that the new action $\tilde{D}$ constructed from $D$ in Remark~\ref{rem:l5_top} will have the following description.

\begin{lemma} \label{lem:self_comp_ds}
Let $D$ be an $(r,s)$-self compatible $C_n$-action as in Definition~\ref{def:self_comp_ds}. Then
 the tuple $$\tilde{D}=(n,g_0+1;(c_{1},n_{1}),\dots,\widehat{(c_{r},n_{r})}, \ldots, \widehat{(c_{r},n_{s})}, 
 \dots, (c_{\ell},n_{\ell}))$$ defines a $C_n$-action such that 
 $$g(\tilde{D}) = g(D)+n/n_r \text{ and }g_0(\tilde{D}) = g_0(D)+1.$$
\end{lemma}

\noindent We denote the action $\tilde{D}$ described in Lemma~\ref{lem:self_comp_ds} by $\l D,(r,s)\r$. We will see later that this kind of compatibility will hold great relevance in the realization of Type 2 actions.

\begin{remark}
\label{rem:triv_self_comp}
It is quite apparent that the construction in Remark~\ref{rem:l5_top} is always possible, if one assumes that the sizes of the pair of chosen orbits equals $n$. We will call such a compatibility a \textit{trivial self compatibility}. While the new $C_n$-action $\tilde{D}$ in this case will have the same fixed point data as $D$, the construction increases the genus of the quotient orbifold by one, and so we have:
$$\tilde{D} = (n,g_0+1; (c_1,n_1), (c_2,n_2),\ldots, (c_{\ell},n_{\ell})).$$ We will denote the 
action $\tilde{D}$ describe above (with $g(\tilde{D}) = g(D) + n$) by $\l D\r$, and we will denote the data set $\tilde{D}$ obtained from $D$ by an inductive application of the trivial self compatibility $g'$ times by $\l D,g' \r$. Note that $\l D,g' \r$ is topologically equivalent to the action obtained from $D$ by removing cyclically permuted (mutually disjoint) disks around points in an orbit of size $n$, and then attaching $n$ copies of the surface $S_{g',1}$ along the resultant boundary components. (Here, $S_{g,b}$ denotes the surface of genus $g$ with $b$ boundary components.) Conversely, given an action of type $\tilde{D} = \l D, g'\r$, for some $g' \geq 1$, one can reverse the topological process described above to obtain the action $D$.
\end{remark}

 The following theorem, which is simply a formalization of the $\l D,g' \r$ type construction described in Remark~\ref{rem:triv_self_comp}, will play a central role in the geometric realizations of arbitrary Type 1 actions.

\begin{theorem}\label{lem:sph_act_to_g}
Every Type 1 action $\tilde{D}$ with $g_0(\tilde{D}) > 0$ is of the form $\tilde{D} = \l D,  g_0(\tilde{D})\r$.
\end{theorem}

\noindent The following corollary, which gives an explicit topological realization for generic Type 1 actions, follows directly from Theorems~\ref{main} and~\ref{lem:sph_act_to_g}. 

\begin{corollary}
For $g \geq 2$, a Type 1 action $D$ on $S_g$ with $g_0(D) > 0$ can be realized explicitly as the rotation by $\theta_D = 2\pi c_3^{-1}/n_3$ of a hyperbolic polygon $\p_D$ with a suitable side-pairing $W(\p_D)$, where $\p_D$ is a hyperbolic  $k(D)$-gon with
$$ \small k(D) := \begin{cases}
2n(1+2g_0(D)), & \text { if } n_1,n_2 \neq 2, \text{ and } \\
n(1+4g_0(D)), & \text{otherwise, }
\end{cases}$$
and for $0 \leq m\leq n-1$, 
$$ \small
W(\p_D) =
\begin{cases}
\displaystyle  
  \prod_{i=1}^{n} Q_ia_{2i-1} a_{2i} \text{ with } a_{2m+1}^{-1}\sim a_{2z}, & \text{if } k(D) = 2n, \text{ and } \\
\displaystyle
 \prod_{i=1}^{n} Q_ia_{i} \text{ with } a_{m+1}^{-1}\sim a_{z}, & \text{otherwise,}
\end{cases}$$
where $\displaystyle z \equiv m+qj \pmod{n}, \,q= (n/n_2)c^{-1},j=n_{2}-c_{2}$, and \\ $\small Q_r = \prod_{s=1}^{g_0(D)} [x_{r,s},y_{r,s}], \, 1 \leq r \leq n.$
\end{corollary} 

\noindent Given a topological realization of an action $D$ on $S_{g(D)}$, in the following remark, we describe explicit hyperbolic structures on $S_{g(\tilde{D})}$ that realize the actions of types $\tilde{D} = \l D,(r,s) \r $ and $\tilde{D} = \l D, g'\r$. 

\begin{remark}\label{rem:l5_geom}
Consider a $(r,s)$-self compatible $C_n$-action $$D = (n,g_0;(c_1,n_1),\ldots,(c_{\ell},n_{\ell})),$$ and a geometric structure $\G(D)$ on $S_{g(D)}$ that realizes $D$ as an isometry, and let $k = n/n_r$. For $1 \leq i \leq k$, we choose $k$ pairs $O_i,O_i'$ of distinct orbits of $D$ of size $n$. For each $i$, we choose disjoint disks $B_{i,j}$ and $B_{i,j}'$, for $1 \leq j \leq n$, of the same area $a$ around the points in $O_i$ and $O_i'$, respectively, that are cyclically permuted by $D$. After removing these disks, for each $i$, we identify the $k$ pairs $\partial B_{i,j}, \partial B_{i,j}'$ of the resulting boundary components of length $\ell = \ell(a)$ (without any twisting), thereby defining a hyperbolic structure 
$$\G(\l D, (r,s)\r, \ell)$$ on $S_{g(\tilde{D})}$ that realizes $\tilde{D} = \l D, (r,s)\r$, which depends on $\G(D)$ and the parameter $\ell$. Conversely, given a hyperbolic structure of type $\G(\l D, (r,s)\r, \ell)$ on $S_{g(\tilde{D})}$ (for suitably chosen $\ell$), we can reverse this geometric construction  to obtain the structure $\G(D)$ on $S_{g(D)}$. In an analogous manner, we can define a structure $\G(\l D \r, \ell)$ that realizes an action of type $\l D \r$, which when applied inductively yields a structure that realizes a $\l D,g'\r$ type action. This structure, which depends on $\G(D)$ and $g'$ parameters $\ell_i, \,  1 \leq i \leq g'$,  will be denoted by $\G(\l D ,g' \r , (\ell_1, \ldots,\ell_{g'}))$.
\end{remark} 

We have shown that for a Type 1 action $\tilde{D}$ with $g_0(\tilde{D}) > 0$, one can apply Theorem~\ref{lem:sph_act_to_g} to get another Type 1 action $D$ with identical fixed point data with $g_0(D) = 0$. We can then use Theorem~\ref{main} to obtain a structure $\G(D)$ that realizes $D$. Finally, by Remarks~\ref{rem:l5_top},~\ref{rem:triv_self_comp}, and~\ref{rem:l5_geom}, we can construct a hyperbolic structure of type $\G(\l D, g'\r, (\ell_1,\ldots,\ell_{g'}))$ that realizes $\tilde{D}$. Thus, it follows immediately that a structure of type $\G(\l D, g'\r, (\ell_1,\ldots,\ell_{g'}))$ realizes an arbitrary Type 1 action, which we formalize in the following theorem.

 \begin{theorem} \label{thm:real_type2}
Consider a Type 1 action $\tilde{D}$ with $g_0(\tilde{D}) = g'> 0$. Then there exists a Type 1 action $D$ with $g_0(D) = 0$, and positive real numbers $\ell_1, \ldots, \ell_{g'}$ such that $\tilde{D}$ is realizable as an isometry of the surface $S_{g(\tilde{D})}$ with the hyperbolic structure $\G(\l D ,g' \r , (\ell_1, \ldots,\ell_{g'}))$.
\end{theorem}

\subsection{Type 2 actions}

We motivate the realization of Type 2 actions with an interesting example. 

\begin{example}\label{eg:type-3}
Consider the $C_6$-action 
$$D= (6,0;(1,2),(1,2),(1,3),(2,3))$$ on $S_2$.
Observe that $D$ has no cone point of order $6$, and hence $D$ is a Type 2 action. However, we can realize the action $D$ in the following manner. Consider the Type 1 $C_6$-actions $$(6,0;(1,2),(1,3),(1,6))\text{ and } (6,0;(1,2),(2,3),(5,6))$$ on $S_1$. The local rotational angles induced by these actions around their unique fixed points are $2\pi/6$ and $10 \pi/6$, respectively, which add up to $0 \pmod{2\pi}$. Now consider these actions on two distinct copies of $S_1$, remove invariant disks around the fixed point on each copy of $S_1$, and then attach an annulus across the resultant boundary circles, as illustrated in Figure~\ref{fig:3}, thereby realizing the action $D$.

\begin{figure}[H]
\labellist
\small
\pinlabel $2\pi/6$ at 49 21
\pinlabel $10\pi/6$ at 99 21
\pinlabel $D_1\curvearrowright$ at 1 39
\pinlabel $\curvearrowleft D_2$ at 147 39
\endlabellist
\centering
\includegraphics[width = 55 ex]{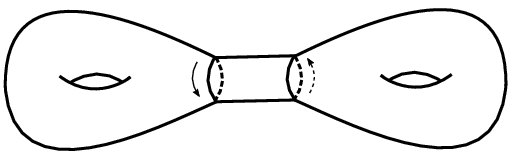}
\caption{A realization of the action $D$ on $S_2$.}
\label{fig:3}
\end{figure}
\end{example}

\noindent This realization of $D$ was made possible by the fact that the rotational angles around the fixed points of the actions added up to $0$ modulo $2\pi$. Taking inspiration from this idea, we now define the notion of a compatible pair of Type 1 actions, that is, a pair of Type 1 actions that will allow such a construction. 

\begin{definition}\label{def:comp_pair}
For $i = 1,2$, two actions
$$D_{i}=(n, g_{i,0}; (c_{i,1} , n_{i,1} ),(c_{i,2},n_{i,2}),\ldots,(c_{i,\ell_i},n_{i,\ell_i}))$$
are said to form an $(r,s)$-\textit{compatible pair} $(D_1,D_2)$ if there exists $1 \leq r \leq \ell_1$ and $ 1 \leq s \leq \ell_2$ such that 
\begin{enumerate}[(i)]
\item $n_{1,r} = n_{2,s} = m$, and
\item $c_{1,r}+c_{2,s} \equiv 0 \pmod{m}$. 
\end{enumerate}
The number $1+g(D) - g(D_1) - g(D_2)$ will be denoted by $A(D_1,D_2).$
\end{definition}

In the following lemma, we give a combinatorial recipe for constructing a Type 2 action from compatible pair of Type 1 actions. The proof of the lemma follows immediately from the definition of a data set.

\begin{lemma} \label{lem:comp_pair}
Given an $(r,s)$-compatible pair $(D_1,D_2)$ of actions, we obtain an action 
\begin{gather*}
 D=(n,g_{1,0}+g_{2,0};(c_{1,1},n_{1,1}),\dots,\widehat{(c_{1,r},n_{1,r})}, \ldots, (c_{1,\ell_1},n_{1,\ell_1}),-\\ 
 (c_{2,1},n_{2,1}),\dots,\widehat{(c_{2,s},n_{2,s})}, \ldots, (c_{2,\ell_2},n_{2,\ell_2})
\end{gather*}
such that $$A(D_1,D_2) = \frac{n}{n_{1,r}}.$$
\end{lemma}

 \noindent  We denote action $D$ in Lemma~\ref{lem:comp_pair} obtained from an $(r,s)$-compatible pair $(D_1,D_2)$ by $\lp D_1,D_2,(r,s) \rp$. The geometric structure that realizes a Type 2 action $D= \lp D_1,D_2,(r,s) \rp$, where the $D_i$ are Type 1 actions (as Lemma~\ref{lem:comp_pair} indicates) will depend on the individual structures $\G(D_i)$, and the compatibility. We make this more explicit in the following remark. 
 
 \begin{remark}\label{rem:geom_comp_pair}
Let $D$ be a Type 2 action that can be realized as a $(r,s)$-compatible pair $(D_1,D_2)$ of Type 1 actions as in Lemma~\ref{lem:comp_pair}. On each $S_{g(D_i)}$, the $D_i$ will induce an orbit $O_i$ of size $b = n/n_{1,r}$. Furthermore, $D_1$ and $D_2$ induce local rotational angles $\theta_1 = 2 \pi c_{1,r}^{-1}/n_{1,r}$ and $ \theta_2 = 2 \pi c_{2,s}^{-1}/n_{2,s}$ around the points in $O_1$ and $O_2$ respectively. For $1 \leq j \leq b$, we choose mutually disjoint disks $B_{i,j}$ around each point in the $O_i$ that are cyclically permuted by the $D_i$, ensuring that the $b$ pairs $B_{1,j},B_{2,j}$ of invariant disks have the same area $a$. We remove these disks on each $S_{g(D_i)}$ yielding the surfaces $\Gamma_i \approx S_{g(D_i),b}$, all of whose boundary components have the same length $\ell = \ell(a)$. On each $\Gamma_i$, there is a natural hyperbolic structure $\G_i$ induced by the structures $\G(D_i)$. Since condition $(ii)$ of Definition~\ref{def:comp_pair} on the pair $(D_1,D_2)$ ensures that $\theta_1+\theta_2 \equiv 0 \pmod{2\pi},$ we can identify the $b$ pairs $\partial B_{1,j}, \partial B_{2,j}$ of boundary circles, thereby extending the actions $D_i$ on the $S_{g(D_i)}$ to the action $D$ on $S_{g(D)}$. Note that this identification also defines an hyperbolic structure (that we denote by) $\G(\lp D_1,D_2, (r,s)\rp,\ell)$ on $S_{g(D)}$ realizing $D$, that is completely determined by the structures $\G(D_i)$, the number $A(D_1,D_2)$, and the parameter $\ell$.  
\end{remark}

\noindent The discussion in Remark~\ref{rem:geom_comp_pair} leads to the following theorem. 
\begin{theorem}\label{thm:comp_pair}
Let $D$ be a Type 2 action that can be realized as an $(r,s)$-compatible pair $(D_1,D_2)$ of Type 1 actions as in Lemma~\ref{lem:comp_pair}. Then there exists $\ell > 0$ such that $D$ can be realized as an isometry of the hyperbolic structure $\G( \lp D_1,D_2, (r,s)\rp,\ell)$ on $S_{g(D)}$.
\end{theorem}

\begin{remark}\label{rem:triv_comp_pair}
In the context of $(r,s)$-compatible pairs of actions, there exists a natural analog to the self compatibility that we have seen in Remark~\ref{rem:triv_self_comp}. As expected, this compatibility also occurs across a pair of full-sized orbits under the actions. For $i = 1,2$, consider actions 
$$D_{i}=(n, g_{i,0}; (c_{i,1} , n_{i,1} ),(c_{i,2},n_{i,2}),\ldots,(c_{i,\ell_i},n_{i,\ell_i}))$$ of degree $n$. Choose distinguished orbits $O_i$ of $D_i$ of size $n$ on $S_{g(D_i)}$. For each $i$, we choose mutually disjoint disks $B_{i,j}$, for $1 \leq j \leq n$ having the same area $a$ around the points in $O_i$ that are cyclically permuted by the $D_i$. We remove these disks yielding the surfaces $\Gamma_i = S_{g(D_i),n}$ all of whose boundary components have the same length $\ell$. As there is no local rotation induced around each component of $\partial \Gamma_i$, we identify the $n$ pairs $B_{1,j},B_{2,j}$ of boundary circles. This process yields a $C_n$-action 
\begin{gather*}
 D = \lp D_1,D_2\rp :=(n,g_{1,0}+g_{2,0};(c_{1,1},n_{1,1}),\ldots, (c_{1,\ell_1},n_{1,\ell_1}),-\\ 
 (c_{2,1},n_{2,1}), \ldots, (c_{2,\ell_2},n_{2,\ell_2}))
\end{gather*}
with $g(D) = g(D_1)+g(D_2) + n-1$, and a hyperbolic structure (we denote by) $\G (\lp D_1,D_2 \rp, \ell)$ on $S_{g(D)}$ that is completely determined by the structures $\G(D_i)$ and parameter $\ell$. 
\end{remark}

\noindent In order to obtain the realizations of arbitrary Type 2 actions, we need to understand the realizations of irreducible Type 2 actions. The realizations of such actions is rather subtle, which we address through the following technical lemma. 

\begin{lemma}\label{lem:type3_irr}
Let $D$ be an irreducible Type 2 action. Then $\l D ,1 \r$ can be built topologically from 
three pairwise compatible irreducible Type 1 actions.
\end{lemma}

\begin{proof}
Let $$D=(n, 0; (c_1,n_1), (c_2,n_2), (c_3,n_3)),$$ where $n_i\neq n$, for any $i$, be a Type 2 action. Without loss of generality, we may assume that $n/n_1,n/n_2$ are odd and $n/n_3$ is even. By a simple number-theoretic argument, we can see that for $i = 1,2,3$, the tuple
\begin{gather*}
D_i=(n, 0; (c_{i,1} , n_{i,1} ),(c_{i,2},n_{i,2}),(c_{i,3}, n_{i,3})),\text{ where } \\
n_{i,3}=n_{3,2}=n, \, (c_{i,1}, n_{i,1})=(c_i, n_i),\text{ and }n_{2,2}=n_{1,2}=n_3, 
\end{gather*}
is a Type 2 data set. Moreover, we can choose our $c_{i,j}$'s in order to ensure that $(D_1,D_2)$ forms a $(2,2)$-compatible pair, and $(D_2,D_3)$ forms a $(3,3)$-compatible pair. These compatibilities yield a $(4,5)$-self compatible Type 2 data set 
$$\tilde{D}=(n, 0, (c_1,n_1), (c_2,n_2), (c_3,n_3), (c_4,n_4), (c_5, n_5)),$$ with $g(\tilde{D}) = g(D)+n-1$.
Finally, we apply Lemma~\ref{lem:self_comp_ds} to obtain a topological realization of the action
$$(n, 1; (c_1,n_1), (c_2,n_2), (c_3,n_3)),$$ from which the result follows.
\end{proof}

\noindent We have now developed the requisite tools to describe the realizations of arbitrary Type 2 actions. 

\begin{theorem}\label{thm:arb_real}
For $g \geq 2$, a Type 2 action on $S_g$ can be constructed from finitely many compatibilities of the following types involving Type 1 actions:
\begin{enumerate}[(i)]
\item $\l D, (r,s)\r$,
\item $\l D \r$,
\item $\lp (D_1,D_2), (r,s)\rp$, and
\item $\lp (D_1,D_2) \rp$.
\end{enumerate}
\end{theorem}

\begin{proof}
For a Type 2 action $D$ of degree $n$ on $S_g$, it follows from the work of J. Gilman~\cite{JG3} that $h$ has a nonempty reduction system $\C$. We may assume, without loss of generality, that $\C$ is maximal. Let $S_g(\C)$ denote the surface obtained from $S_g$ by removing a closed annular neighborhood $N$ of $\C$ and then capping $\overline{S_g\setminus N}$. Note that every nonseparating $c \in \C$ yields $2$ distinguished marked points on the same component of $S_g(\C)$, while a separating curve $c' \in \C$ yields two distinguished marked points on two distinct components of $S_g(\C)$. Clearly, $h$ induces a $C_n$-action $\tilde{h}$ on $S_g(\C)$. If every component of $S_g(\C)$ is the $2$-sphere $S_0$, then it is apparent that $h$ is a rotational action, which contradicts our assumption. Hence, $S_g(\C)$ has at least one component $S_{g'}$, for some $0 < g' \leq g$, and every component $F$ of this form satisfies $\tilde{h}(F) =F$ with $\tilde{h}\vert_{F} \in \Mod(F)$ being irreducible. Moreover, there exists no spherical component $F'$ of $S_g(\C)$ such that $(\tilde{h}\vert_{F'})(F') = F'$

Suppose that $S_g(\C)$ has only one component, which is not a sphere. Then it is follows from~\cite{JG3} that $D$ has to be an irreducible Type 2 action. By Theorem~\ref{lem:sph_act_to_g}, it follows that there exits an action of type $\tilde{D} = \l D,1\r$ with $g_0(\tilde{D}) = 1$. We now appeal to Lemma~\ref{lem:type3_irr} to realize $\tilde{D}$ as a combination of two $(r,s)$-compatibilities, and one self-compatibility. If $S_g(\C)$ has exactly two nonspherical components say $F_i$ so that $\tilde{h} \vert_{F_i} = D_i$, for $i = 1,2$. When $S_g(\C)$ has no spherical components, $h$ is realized either by a structure of type $\lp (D_1,D_2), (r,s) \rp$ or $\lp (D_1,D_2) \rp$. On the other hand, if $S_g(\C)$ has a nontrival spherical orbit, then this orbit should have come from a $\l D , g'\r$-type compatibility (with $D = D_1 \text{ or } D_2$) for some $g' > 0$. Applying this approach inductively, we can realize all Type 2 actions. 
\end{proof}

\noindent The following corollary is a direct consequence of Theorem~\ref{thm:arb_real}, and  Remarks~\ref{rem:l5_geom},~\ref{rem:geom_comp_pair}, and~\ref{rem:triv_comp_pair}.

\begin{corollary}\label{cor:arb_real}
For $g \geq 2$, the hyperbolic structure that realizes an arbitrary Type 2 action on $S_g$ as an isometry can be constructed from finitely many structures of the following types:
\begin{enumerate}[(i)]
\item $\G(\l D, (r,s)\r, \ell)$, 
\item $\G(\l D \r, \ell)$,
\item $\G(\lp (D_1,D_2), (r,s)\rp, \ell)$, and
\item $\G(\lp (D_1,D_2) \rp,\ell)$,
\end{enumerate}
\end{corollary}

\noindent We conclude this section with an example of a topological realization of an arbitrary Type 2 action whose quotient orbifold has more than 3 cone points. The geometric structure that realizes this action is implicit from the discussions that we have seen so far.
\begin{example}

Consider the $C_{30}$-action 
 $$D=(30, 1; (1,2), (1,6), (1,10), (7,30))$$ on $S_{49}$. We will construct this action from the following Type 1 actions, namely
\begin{gather*}
D_1=(30, 0; (11,15), (19,30), (19,30)), D_2=(30, 0; (1,6), (7,15), (11,30)), \\
D_3=(30, 0; (1,10), (8,15), (11,30)), \text{ and } D_4=(30, 0; (1,2), (4,15), (7,30)).
\end{gather*}
  Note that  $(D_1,D_2)$ forms a $(3,3)$-compatible pair, and $(D_2,D_3)$ forms a $(2,2)$-compatible pair. Using these compatibilities, By Lemma~\ref{lem:type3_irr}, we obtain the following $C_{30}$-action on $S_{41}$
 $$D_5 =(30, 1, (1,6), (1,10), (11,15)).$$ 
 Finally, observe that $(D_4,D_5)$ forms a $(2,3)$-compatible pair, thereby realizing the required action. 
\end{example}

\section{A combinatorial perspective using Fat Graphs}\label{sec:3}
A fat graph~\cite{BS1} is a graph equipped with a cyclic order on the set of edges incident at each vertex. If the valency of a vertex is less than three, then the cyclic order on the set of edges incident at that vertex is trivial. So, we consider the graphs with valency at each vertex is at least three. We denote a graph by a triple $G=(E, \sim, \sigma_1)$, where $E$ is a non-empty, finite set with even number of elements, $\sim$ is an equivalence relation on $E$ and $\sigma_1$ is a fixed-point free involution on $E$. It is straight forward to see that $G$ is equivalent to a standard graph. Each element of the set $E_1=E/\sigma_1$ is called an undirected edge. The set of vertices is $V=E/\sim$. If $v\in V$ is a vertex then the degree of $v$ is defined by $deg(v)= |v|$.

\begin{definition}
A \textit{fat graph} is a quadruple $G=(E, \sim, \sigma_1, \sigma_0)$, where
\begin{enumerate}[(i)] 
\item $(E, \sim, \sigma_1)$ is a graph, and 
\item $\sigma_0$ is a permutation on $E$ so that each cycle of $\sigma_0$ is a cyclic order on some $v\in V$.
\end{enumerate}
\end{definition}

\begin{definition}
A bijective map $f: (E, \sim, \sigma_1, \sigma_0) \to (E', \sim', \sigma_1', \sigma_0')$ is said to be a \textit{fat graph isomorphism} if 
\begin{enumerate}[(i)]
\item $x \sim y \iff f(x) \sim' f(y)$, for all $x,y \in E$, and
\item $f \circ \sigma_i = \sigma_i' \circ f, \text{ for } i = 0,1.$
\end{enumerate}
We denote the automorphism group of a fat graph $\Gamma$ by $\text{Aut}(\Gamma)$.
\end{definition}

\begin{example}
Consider the quadruple $G=(E, \sim, \sigma_1, \sigma_0)$ described in the following manner (see Figure~\ref{fig:12}). The set of directed edges is $E=\{\vec{e}_i, \cev{e}_i|\ i=1,2,3\}.$ The equivalence classes of $\sim$ are $v_1=\{\vec{e}_i|\ i=1,2,3\}$ and $v_2=\{\cev{e}_i|\ i=1,2,3\}.$ The fixed point free involution $\sigma_1$ is given by $\sigma_1 = (\vec{e}_1, \cev{e}_1)(\vec{e}_2, \cev{e}_2)(\vec{e}_3, \cev{e}_3)$ and $\sigma_0=(\vec{e}_1, \vec{e}_2, \vec{e}_3)(\cev{e}_1, \cev{e}_3, \cev{e}_2).$ The surface obtained by thickening the edges of the graph (as shown in Figure~\ref{fig:12}) is the sphere with three holes. Note that, if one considers  $\sigma'_0=(\vec{e}_1, \vec{e}_2, \vec{e}_3)(\cev{e}_1,\cev{e}_2, \cev{e}_3)$ instead of $\sigma_0$, then the associated surface is $S_{1,1}$.

\tikzset{->-/.style={decoration={
  markings,
  mark=at position .5 with {\arrow{>}}},postaction={decorate}}}
 
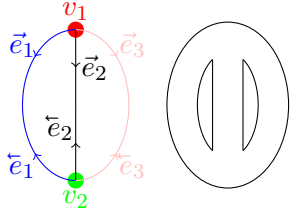
\begin{figure}[H]
\begin{center}
\begin{tikzpicture}
\begin{scope}[every node/.style={sloped,allow upside down}]

\draw [fill,red] (0,1) circle [radius=0.1]; \draw [fill,green] (0,-1) circle [radius=0.1]; \draw [->-] (0,1) -- (0, 0); \draw [->-] (0,-1) -- (0, 0);  \draw (0.25, 0.5) node {$\vec{e}_2$}; \draw (-0.2, -0.3) node {$ \cev{e}_2$};

\draw [->-,blue,domain=90:180] plot ({0.7*cos(\x)}, {sin(\x)}); \draw [->-, blue,domain=180:270] plot ({0.7*cos(90-\x)}, {sin(90-\x)});

\draw [blue] (-0.7, -0.8) node {$\cev{e}_1$}; \draw [pink] (0.75, -0.8) node {$\cev{e}_3$};\draw [blue] (-0.7, 0.8) node {$\vec{e}_1$}; \draw [pink] (0.75, 0.8) node {$\vec{e}_3$}; \draw [red] (0, 1) node [above] {$v_1$}; \draw [green] (0,-1) node [below] {$v_2$};

\draw [pink, ->-,domain=0:90] plot ({0.7*cos(90-\x)}, {sin(90-\x)}); \draw [->-, pink,domain=270:360] plot ({0.7*cos(\x)}, {sin(\x)});

\draw (2,0) ellipse (0.8cm and 1.1cm); \draw (1.8,-0.6) -- (1.8, 0.6); \draw (2.2,-0.6) -- (2.2, 0.6);   
\draw [domain=120:240] plot ({2+0.4*cos(\x)}, {0.7*sin(\x)});\draw [domain=300:420] plot ({2+0.4*cos(\x)}, {0.7*sin(\x)});

\end{scope}
\end{tikzpicture}
\caption{A fat graph with three boundary components.}\label{fig:12}
\end{center}
\end{figure}
\end{example}

\noindent A fat graph is called \textit{decorated} if the degree of each vertex is an even integer $\geq 4$. A cycle in a decorated fat graph is called a \textit{standard cycle} if every two consecutive edges are opposite to each other in the cyclic order on the edges incident at the their common vertex. Note that each edge uniquely determines a standard cycle. A decorated fat graph can be written as an edge-disjoint union of standard cycles. Let $\Sigma(G)$ be the surface with boundary associated with $G$. Then the number of boundary components in $\Sigma(G)$ is the same as the number of orbits of $\sigma_1\sigma_0^{-1}$ (see~\cite{BS1} for more details on this topic).

\subsection{Fat graph for a polygon with side-pairing} 
Given a polygon with a side-pairing, we can associate a fat graph with it. Let $$A= \{a_1, \dots, a_k, a_1^{-1}, \dots, a_k^{-1}\},$$ and $\p = \p_{2k}$ with a boundary word $W(\p)$ comprising letters from $A$. Then we define a fat graph $\Gamma(\p)=(E, \sim, \sigma_1, \sigma_0)$ associated with the polygon $\p$ as follows:
\begin{enumerate}[(i)]
\item $E=\{a_i, a_i^{-1}|\ i=1,2,\dots, k\}$.
\item Before describing the equivalence relation $\sim$, we describe $\sigma_1$. The fixed point free involution is defined by $\sigma_1(a_i)=a_i^{-1}$, and $\sigma_1(a_i^{-1})=a_i$, for $i=1,2,\dots, k.$
\item The equivalence classes of $\sim$ are defined in the following manner. Given $x_1\in E$, we define $x_2$ to be the edge of $\p$ immediately after $\sigma_1(x_1)$. In general, we define $x_{i+1}$ to be the edge of $\p$ immediately after $\sigma_1(x_i)$. Then the equivalence class of $x_1$ is $\{x_1, x_2, \dots x_n\}$, where $n$ is the smallest positive integer such that $x_{n+1}=x_1.$
\item If $v=\{x_1, x_2, \dots, x_n\}$ is a vertex such that $x_{i+1}$ is the edge of $\p$ immediately after $\sigma_1(x_i)$ for all $i$, then we define $c_v=(x_1, x_2,\dots, x_n)$. The fat graph structure $\sigma_0$ is given by $$\sigma_0=\prod\limits_{v\in V}c_v.$$ 
\end{enumerate}

\begin{example}
Consider the polygon $\p = \p_{10}$ (see Figure~\ref{fig:13}) labeled by the elements of $A\!=\!\{a,b,c,d,e, a^{-1}, b^{-1}, c^{-1}, d^{-1}, e^{-1}\}.$ Then the fat graph $\Gamma(\p)=(E, \sim, \sigma_1, \sigma_0)$ is described as follows.
\begin{enumerate}[(i)]
\item $E=\{a,b,c,d,e, a^{-1}, b^{-1}, c^{-1}, d^{-1}, e^{-1}\}$.
\item $\sigma_1(a)=a^{-1}, \, \sigma_1(b)=b^{-1}, \, \sigma_1(c)=c^{-1}, \, \sigma_1(d)=d^{-1}, \text{ and }\sigma_1(e)=e^{-1}$.
\item Let $x_1=a$, then $x_2=b^{-1}$ which is the edge of $\p$ next to $\sigma_1(a)=a^{-1}$. Similarly, we have, $x_3=c, \, x_4=d^{-1}, \, x_5=e$, and $x_6=a$. Therefore, the equivalence class containing $a$ is $v_1=\{x_1, x_2, x_3, x_4, x_5\}= \{a, b^{-1}, c, d^{-1}, e\}.$ The other equivalence class is $v_2=\{b, c^{-1}, d, e^{-1}, a^{-1}\}.$  Thus, $V=E/\!\!\sim = \{v_1, v_2\}$.
\item $\sigma_0=(a, b^{-1}, c, d^{-1}, e)(b, c^{-1}, d, e^{-1}, a^{-1})=(a, b^{-1}, c, d^{-1}, e)(a^{-1}, b, c^{-1}, d, e^{-1})$.
\end{enumerate}
\end{example}
\tikzset{->-/.style={decoration={
  markings,
  mark=at position .5 with {\arrow{>}}},postaction={decorate}}}
 
\begin{figure}[H]
\begin{center}
\begin{tikzpicture}
\begin{scope}[every node/.style={sloped,allow upside down}]

\draw [fill,green] (2,0) circle [radius=0.05]; \draw [fill,red] (-2,0) circle [radius=0.05]; \draw [fill,red] (1.62,1.17) circle [radius=0.05]; \draw [fill,green] (-1.62,1.17) circle [radius=0.05]; \draw [fill,green] (-1.62,-1.17) circle [radius=0.05]; \draw [fill,red] (1.62,-1.17) circle [radius=0.05];  \draw [fill,green] (0.62,1.9) circle [radius=0.05]; \draw [fill,red] (-0.62,1.9) circle [radius=0.05]; \draw [fill,red] (-0.62,-1.9) circle [radius=0.05]; \draw [fill,green] (0.62,-1.9) circle [radius=0.05]; 

\draw [->-] (-0.62, -1.9) -- (0.6, -1.9); \draw [->-] (0.62, -1.9) -- (1.62, -1.17); \draw (0,-1.7) node {$a$}; \draw (1,-1.35) node {$b$}; \draw [->-] (1.62, -1.17) -- (2, 0); \draw (1.65,-0.5) node {$c$}; \draw [->-] (2, 0) -- (1.62, 1.17); \draw (1.65,0.5) node {$d$}; \draw [->-] (1.62, 1.17) -- (0.62, 1.9); \draw (1,1.35) node {$e$}; \draw [->-] (0.6, 1.9) -- (-0.62, 1.9); \draw (0,1.7) node {$a^{-1}$}; \draw [->-] (-0.62, 1.9) -- (-1.62, 1.17); \draw (-0.9,1.35) node {$b^{-1}$}; \draw [->-] (-1.62, 1.17) -- (-2, 0); \draw (-1.45,0.5) node {$c^{-1}$}; \draw [->-] (-2, 0) -- (-1.62, -1.17); \draw (-1.4,-0.5) node {$d^{-1}$}; \draw [->-] (-1.62, -1.17) -- (-0.62, -1.9); \draw (-0.8,-1.35) node {$e^{-1}$};

\draw [blue, ->-,->-] (0,0) circle [radius = 0.5];

\draw [fill,red] (5.25,1) circle [radius=0.05]; \draw [fill,red] (5,1.25) circle [radius=0.05]; \draw [fill,red] (4.75,1) circle [radius=0.05]; \draw [fill,red] (5.125,0.75) circle [radius=0.05]; \draw [fill,red] (4.875,0.75) circle [radius=0.05];

\draw [fill,green] (5.25,-1) circle [radius=0.05]; \draw [fill,green] (5,-1.25) circle [radius=0.05]; \draw [fill,green] (4.75,-1) circle [radius=0.05]; \draw [fill,green] (5.125,-0.75) circle [radius=0.05]; \draw [fill,green] (4.875,-0.75) circle [radius=0.05];

\draw [->-](5.125,-0.7) -- (5.125,0.7); \draw [->-] (4.875,0.7)--(4.875,-0.7);

\draw [rounded corners = 2mm, ->-] (5.175, 0.75) -- (6.25, 0.75) -- (6.25, -0.2) -- (5.125,-0.2); \draw [rounded corners = 3mm, ->-] (5.125,-0.45)-- (6.5, -0.45)-- (6.5, 1) --(5.3, 1); \draw [rounded corners = 1.5mm, ->-] (4.825, -0.75)-- (4.25, -0.75) -- (4.25,-0.45) -- (4.875, -0.45); \draw [rounded corners = 2.5mm, ->-] (4.875, -0.2) -- (4,-0.2) -- (4, -1) -- (4.7, -1);

 \draw [rounded corners=1.5mm, ->-] (5.3, -1)-- (6, -1) --(6, -0.45); \draw  [rounded corners=2mm, ->- ] (6, -0.2) -- (6, 0.45) --(5.125, 0.45); \draw [rounded corners = 1mm, ->-] (4.875, 0.45) -- (4.25, 0.45) -- (4.25, 0.75) -- (4.825, 0.75);
\draw [rounded corners = 2mm, ->-] (4.7, 1) -- (4, 1) -- (4, 0.2) -- (4.875, 0.2); \draw [rounded corners=1.5mm, ->-](5.125, 0.2) -- (5.75, 0.2) -- (5.75, -0.2);\draw [rounded corners=1.5mm, ->-] (5.75, -0.45)-- (5.75, -0.75) -- (5.175, -0.75);

\draw [rounded corners = 2mm, ->-] (4.95, 1.27) -- (4.3, 2) -- (3.25,2) -- (3.25, -1.75) -- (5.5, -1.75) -- (5.55, -1.4) -- (5.27, -1.04);
\draw [->-, rounded corners = 1.2mm](5.05, -1.25) --(5.25, -1.3)-- (5.25, -1.5)--(3.5, -1.5)--(3.5, 1.75)--(4.2, 1.75)--(4.75,1.05);  
\draw [->-, rounded corners = 1.5mm](4.72,-1.04)-- (3.9, -1.3) -- (3.9, -1.5); \draw [->-, rounded corners=1mm](4.15, -1.5) -- (4.15, -1.4) -- (4.95,-1.25);\draw [->-, rounded corners = 2mm] (3.9, -1.75) -- (3.9, -2.2) -- (7, -2.2)-- (7,2) -- (5.7, 2) -- (5.05, 1.27); \draw [rounded corners= 1.5mm, ->-] (5.3, 1.03) -- (5.8, 1.75) --(6.75, 1.75) -- (6.75, -1.95) -- (4.15,-1.95)--(4.15, -1.75);

\draw (5.3, 0) node {\tiny{$a$}}; \draw [blue](4.6, 0) node {\tiny{$a^{-1}$}}; \draw (3.85, -0.5) node {\tiny{$b$}}; \draw (7.15, -0.8) node {\tiny{$c$}}; \draw (3.1, -0.5) node {\tiny{$d$}}; \draw (6, -1.1) node {\tiny{$e$}}; \draw [blue] (4.55, -0.6) node {\tiny{$b^{-1}$}}; \draw [blue] (6.5, -0.7) node {\tiny{$c^{-1}$}}; \draw [blue](3.8, -0.1) node {\tiny{$d^{-1}$}}; \draw [blue](3.8, 0.7) node {\tiny{$e^{-1}$}};

\draw [blue, rounded corners=3mm] (-2.5, 2.2) -- (-2.5, -2.5) -- (11, -2.5) -- (11, 2.2)--cycle; \draw [blue] (7.5, -2.5) -- (7.5, 2.2); \draw [blue] (2.5, -2.5) -- (2.5, 2.2);

\draw [fill,red] (9,1) circle [radius=0.05]; \draw [fill,green] (9,-1) circle [radius=0.05]; \draw (9, 0.95) -- (9, -0.95);

\draw [rounded corners = 1.5mm](8.95, 1) -- (8.5, 1) -- (8.5, 0.4) -- (8.98, 0.4); \draw [rounded corners = 1.5mm](9.03, 0.4) -- (9.5, 0.4) --(9.5, -0.37); \draw [rounded corners=1.5mm](9.5, -0.43) -- (9.5, -1) -- (9.05, -1);
\draw [rounded corners= 1.5mm] (9.05, 1) -- (10, 1) -- (10, -0.4) -- (9.03, -0.4); \draw [rounded corners=1.5mm] (8.97, -0.4) -- (8.5, -0.4) -- (8.5, -1) -- (8.95, -1);

\draw [rounded corners= 1.5mm](9.02, -1.03) -- (9.5, -1.3) -- (9.5, -1.6) -- (8, -1.6) -- (8, 1.5) -- (8.5, 1.5) -- (8.98, 1.03); \draw [rounded corners = 2mm](9.02, 1.03) -- (9.5, 1.5) -- (10.5, 1.5) -- (10.5, -2) -- (8.5, -2) -- (8.5, -1.63); \draw [rounded corners = 1.5mm] (8.5, -1.57) -- (8.5, -1.27) -- (8.98, -1.03);

\draw (9.1,-0.7) node {\tiny{$a$}}; \draw (8.65,-0.87) node {\tiny{$b$}}; \draw (8.7,-1.3) node {\tiny{$c$}}; \draw (9.25,-1.3) node {\tiny{$d$}}; \draw (9.45,-1.05) node {\tiny{$e$}};

\draw (8.9,0.7) node {\tiny{$a$}}; \draw (9.45,0.87) node {\tiny{$b$}}; \draw (9.4,1.27) node {\tiny{$c$}}; \draw (8.8,1.35) node {\tiny{$d$}}; \draw (8.5,1.05) node {\tiny{$e$}};
\end{scope}
\end{tikzpicture}
\caption{A fat graph associated with a decagon with opposite side-pairing.}
\label{fig:13}
\end{center}
\end{figure}
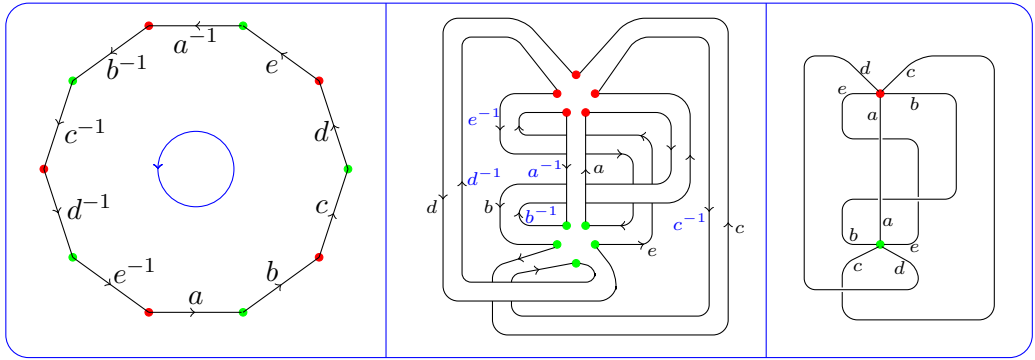

\noindent We will now establish a correspondence between finite order maps realizable as a rotations of a polygons and automorphisms of fat graphs with one boundary component (via $\p \leftrightarrow \Gamma(\p)$). 

\begin{theorem}
\label{thm:fg_auto}
There is a bijective correspondence between finite order homeomorphisms on $S_g$ realizable as rotations of polygons and automorphisms of fat graphs with a single boundary component.
\end{theorem}

\begin{proof}
Let $f$ be a finite order homeomorphism on $S_g$ that is realized as the rotation of a $2k$-sided  polygon $\p$ with boundary word $W(\p)$ comprising letters from $$\{a_1, a_2, \dots, a_k, a_1^{-1}, a_{2}^{-1}, \dots, a_k^{-1}\},$$ which is the set $E$ of directed edges of the graph $\Gamma(\p)$. The map $\tilde{f} : E \to E$ induced by $f$ is injective, as it is the restriction of the rotation map of $\p$. From this, it follows that $\tilde{f}$ is bijective. To prove $\tilde{f}$ is an automorphism, we need to show that the two diagrams in Figure~\ref{fig:14} are commutative.
\tikzset{->-/.style={decoration={
  markings,
  mark=at position .5 with {\arrow{>}}},postaction={decorate}}}
 
\begin{figure}[htbp]
\begin{center}
\begin{tikzpicture}
\begin{scope}[every node/.style={sloped,allow upside down}]
\draw (-1.3,0.6)node {$(1)$};
\draw (-0.2, 0.6)node {$E$};\draw (1.7, 0.6)node {$E$};\draw [->](0, 0.6) -- (1.5, 0.6); \draw (-0.2, -0.6)node {$E$};\draw (1.7, -0.6)node {$E$};\draw [->](0, -0.6) -- (1.5, -0.6); \draw [->](-0.15,0.4) -- (-0.15, -0.4);\draw [->](1.65,0.4) -- (1.65, -0.4);
\draw (3.8, 0.6)node {$E$};\draw (5.7, 0.6)node {$E$};\draw [->](4, 0.6) -- (5.5, 0.6); \draw (3.8, -0.6)node {$E$};\draw (5.7, -0.6)node {$E$};\draw [->](4, -0.6) -- (5.5, -0.6); \draw [->](3.85,0.4) -- (3.85, -0.4);\draw [->](5.65,0.4) -- (5.65, -0.4);
\draw (3,0.6)node {$(2)$};

\draw (0.8,0.85)node {$\tilde{f}$}; \draw (4.8,0.85)node {$\tilde{f}$}; \draw (0.8,-0.88)node {$\tilde{f}$}; \draw (4.8,-0.88)node {$\tilde{f}$};
\draw (-0.4,0)node {$\sigma_1$}; \draw (3.6,0)node {$\sigma_0$}; \draw (1.9,0)node {$\sigma_1$}; \draw (5.9,0)node {$\sigma_0$};
\end{scope}
\end{tikzpicture}
\end{center}
\caption{Commutative diagrams associated with a fat graph.}
\label{fig:14}
\end{figure}
\noindent Let $a\in E$ and $\tilde{f}(a)=a_{*}$. As $\tilde{f}$ is induced from an orientation preserving homeomorphism, we have $\tilde{f}(a^{-1})=\tilde{f}(a)^{-1}=a_{*}^{-1}$. Therefore, we have $\tilde{f}\circ \sigma_1(a)=a_*^{-1}=\sigma_1\circ\tilde{f}(a)$, and so it follows that diagram (1) commutes.

Now, let $a\in E$ and $\tilde{f}(a)=a_*$. By definition, we have $\sigma_0(a)$ is the edge of $\p$ next to $a^{-1}(=\sigma_1(a))$, from which it follows that $\tilde{f}\circ\sigma_0(a)$ is the edge of $\p$ next to $\tilde{f}(a^{-1})$. So, we have $\tilde{f}(a^{-1})=\left(\tilde{f}(a)\right)^{-1}=a_*^{-1}$. Therefore, $\tilde{f}\circ\sigma_0(a)$ is the edge next to $a_*^{-1}$, which is the same as $\sigma(a_*)=\sigma\circ\tilde{f}(a)$, and the commutativity of (2) follows. 

Conversely, let $\Gamma=(E, \sim, \sigma_1, \sigma_0)$ be a fat graph with $v\, (=|E/\!_\sim|)$ vertices, $e \, (=|E/\sigma_1|)$ edges, and a single boundary component satisfying $v-e=1-2g.$  Let $\delta=a_1a_2\ldots a_{2e}$  be the boundary component with $a_i\in E$. Let $F\in \text{Aut}(\Gamma)$ be a fat graph automorphism. We can represent $\delta$ as a reduced word with letters in $E$. In this representation, each letter and its inverse appears exactly once (with the convention that $\left(x^{-1}\right)^{-1}=x$). Therefore, the boundary determines a polygon $\p = \p_{2e}$, and its pairing $W(\p)$. By an Euler characteristic argument  and the condition $v-e=1-2g$, we conclude $S(\p) \approx S_g$. 

The automorphism $F$ naturally extends to an orientation preserving homeomorphism $\widehat{F}$ on the surface $\Sigma(\Gamma)$. The restriction of $\widehat{F}$ on to the boundary is a orientation preserving homeomorphism of the circle, and hence a rotation which we denote by $\tilde{f}$. Now we show that $\tilde{f}$ gives a finite order homeomorphism on $S_g$. Let $e\in E$, then we have $\tilde{f}\circ\sigma_1(e)=\sigma_1\circ \tilde{f} (e)$, from which it follows that $\tilde{f}(e^{-1})=\tilde{f}(e)^{-1}$. The condition $\tilde{f}\circ \sigma_0 = \sigma_0\circ\tilde{f}$ ensures that $\{v_1, \dots, v_s\}$ is an orbit of the vertices of $\p$ with respect to the side-pairing $W(\p)$ if, and only if $\{\tilde{f}(v_1), \dots, \tilde{f}(v_s)\}$ is also an orbit. Therefore, $\tilde{f}$ can be extended to a homeomorphism in the quotient space $S_g$, which we denote by $f$. As $\tilde{f}$ is of finite order, so is $f$.
\end{proof}

\noindent  In order to establish a correspondence between generic finite order mapping classes and fat graph automorphisms, we restrict our attention to a special class of fat graph automorphisms as defined below.  
\begin{definition}
Let $\Gamma=(E,\sim,\sigma_1,\sigma_0)$ be a fat graph of genus $g$ and one boundary component. We say an order $n$ automorphism $F\in \text{Aut}(\Gamma)$ is \textit{irreducible} if it satisfies the following:
\begin{enumerate}[(i)]
\item $|E|=n\text{ or }2n$.
\item $|V / \langle F \rangle| = \begin{cases} 
1, &  \text{if } |E| = n, \text{ and}\\
2, & \text{otherwise.}
\end{cases}$
\end{enumerate} 
\end{definition}

\noindent In the following result, which is a direct consequence of Theorem~\ref{thm:fg_auto}, we establish a correspondence between irreducible fat graph automorphisms and irreducible Type 1 actions. For brevity, we state the result for irreducible fat graph automorphisms $F$ which satisfy $|V / \langle F \rangle|=2$, as the result for the case when $|V / \langle F \rangle|=1$ is analogous.

\begin{corollary}
Let  $\Gamma=(E,\sim,\sigma_1,\sigma_0)$ be a fat graph of genus $g$, and let $F\in \text{Aut}(\Gamma)$ be irreducible of order $n$ with $|V / \langle F \rangle |=2$. Let $F\vert_V=\tau_1\tau_2$ be product of disjoint cycles, where $\tau_i=(v_{i1},\dots,v_{ik_i}),\text{ for }i=1,2$, and $\omega_{ij}$ is the cyclic order at $v_{ij}$. Suppose that $s_i$ is the least positive integer such that $F^{k_i}(\omega_{ij})=(\omega_{ij}^{-1})^{s_i}$, for $i = 1,2$, and $s_3$ is the least positive integer such that $(\sigma_1\sigma_0^{-1})^{s_3}=F$. Then $F$ corresponds to an irreducible Type 1 action $D=(n,0;(c_1,n_1),(c_2,n_2),(c_3,n))$ on $S_g$, where for $i=1,2$, $n_i=n/k_i $, and $c_j s_j\equiv 1\pmod{n_i}$, for $1 \leq j \leq 3$.
\end{corollary}

\noindent In view of above corollary, we denote by $F_D$, the irreducible fat graph automorphism that corresponds to a irreducible Type 1 action $D$, and we have the following definition. 
\begin{definition}\label{def:fat_graph_pair}
We say a pair $F_i$, for $i=1,2$, of irreducible fat graph automorphisms in $\text{Aut}(\Gamma_i)$ form a \textit{compatible pair} $( F_1,F_2 )$, if there exits irreducible Type 1 action $D_i$ such that 
\begin{enumerate}[(i)]
\item $F_i = F_{D_i}$, and 
\item the $D_i$ form an $(r,s)$-compatible pair $(D_1,D_2)$. 
\end{enumerate}
\end{definition}

\noindent We now give an example that will motivate the construction of a fat graph automorphism corresponding to an $(r,s)$-compatible pair.

\begin{example}\label{eg:fg_cpair}
Consider the $C_6$-action 
$$D=(6,0;(1,2),(1,2),(1,6),(5,6))$$ on $S_3$ realizable as a $(2,2)$-compatible pair $(D_1,D_2)$ of the Type 1 data sets, where $$D_1=(6,0;(1,2),(1,3),(1,6))\text{ and }D_2= (6,0;(1,2),(2,3),(5,6)).$$  Applying Theorem~\ref{thm:fg_auto}, we can obtain $F_{D_i}\in \text{Aut}(\Gamma_{i})$, where $\Gamma_{i}=\Gamma(\p_{D_i})$.

Consider the polygons $\tilde{\p_i}$ obtained from $\p_{D_i}$ by removing invariant discs around each point of the compatible orbits of size 2, as shown in Figure~\ref{fig:fat_graph_orbit}. This gives the following gluing conditions (up to some choice):
$$ g_1\sim k_1^{-1},\,g_2\sim k_3^{-1},g_3\sim k_2^{-1},\,h_1\sim l_2^{-1},\,h_2\sim l_1^{-1},\text{ and }h_3\sim l_3^{-1}.$$
\begin{figure}[H]
\labellist
\small
\pinlabel $\huge \curvearrowleft$ at 180 133
\pinlabel $O$ at 180 116
\pinlabel $2\pi/6$ at 180 149
\pinlabel . at 180 129
\pinlabel $e_1$ at 180 -1
\pinlabel $e_1$ at 180 240
\pinlabel $e_2$ at 305 50
\pinlabel $e_3$ at 310 183
\pinlabel $e_3$ at 50 53
\pinlabel $e_2$ at 46 183
\pinlabel $h_2$ at 120 38
\pinlabel $h_3$ at 120 200
\pinlabel $g_1$ at 232 36
\pinlabel $g_3$ at 234 202
\pinlabel $h_1$ at 302 123
\pinlabel $g_2$ at 53 121
\pinlabel $\huge \curvearrowleft$ at 527 133
\pinlabel $O$ at 527 116
\pinlabel $10\pi/6$ at 527 149
\pinlabel . at 527 129
\pinlabel $f_1$ at 527 -3
\pinlabel $f_1$ at 527 242
\pinlabel $f_2$ at 653 52
\pinlabel $f_3$ at 659 183
\pinlabel $f_3$ at 400 52
\pinlabel $f_2$ at 394 183
\pinlabel $l_2$ at 471 36
\pinlabel $l_3$ at 468 200
\pinlabel $k_1$ at 584 38
\pinlabel $k_3$ at 584 204
\pinlabel $l_1$ at 656 121
\pinlabel $k_2$ at 402 121
\endlabellist
\centering
\includegraphics[width = 70 ex]{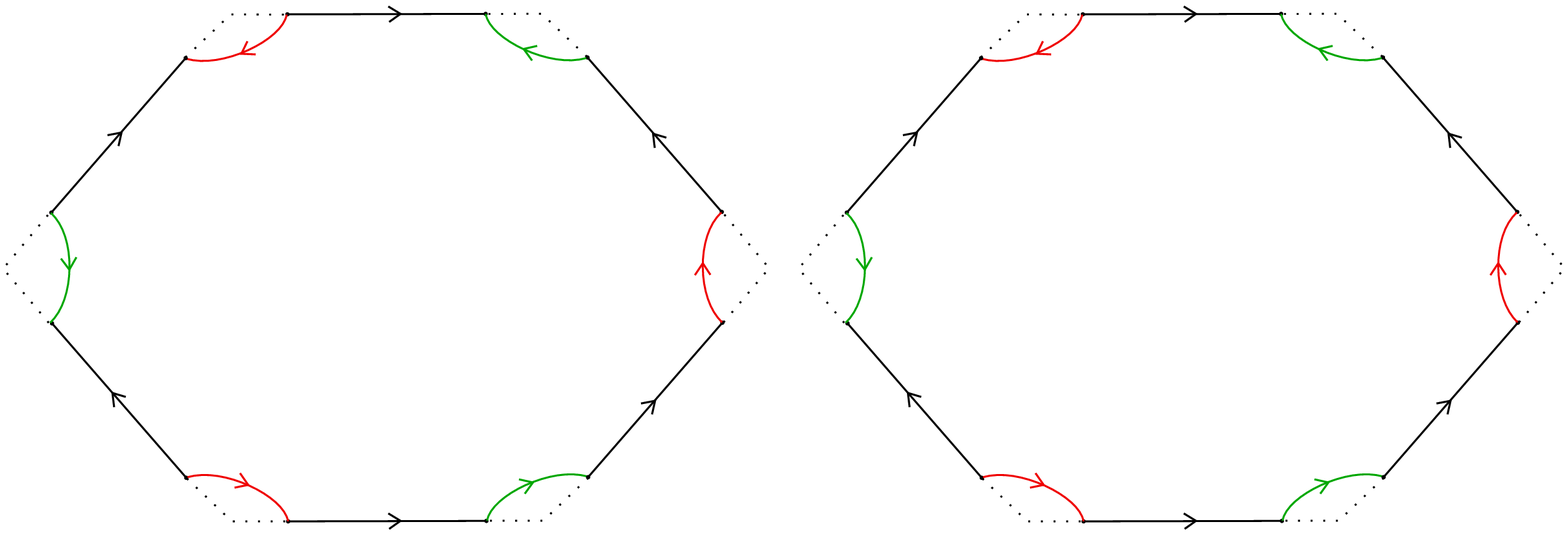}
\caption{The surfaces $\tilde{\p_i}$,     for $i = 1,2$.}
\label{fig:fat_graph_orbit}
\end{figure}
Define a fat graph $\Gamma =(E, \sim, \sigma_1, \sigma_0)$, which is described as follows.
\begin{enumerate}[(i)]
\item $E=\{e_i,f_i,g_i,h_i,e_i^{-1},f_i^{-1},g_i^{-1},h_i^{-1}\mid 1\leq i \leq 3\}$.
\item $\sigma_1(e_i)=e_i^{-1},\sigma_1(f_i)=f_i^{-1}, \sigma_1(g_i)=g_i^{-1},\sigma_1(h_i)=h_i^{-1},\text{ for }1\leq i\leq 3$.
\item Let $x_1=e_1$, then $x_2=h_3$ which is the edge of $\tilde{\p_1}$ next to $\sigma_1(e_1)=e_1^{-1}$.
 By the gluing conditions $h_3^{-1}\sim l_3$, and hence we have, $x_3=f_2^{-1}$,  which is the edge of $\tilde{\p_2}$ next to $\sigma_1(h_3)=h_3^{-1}$.  Again by the gluing conditions, we have $l_1\sim h_2^{-1}$, and since $l_1$ is the edge of $\tilde{\p_2}$ next to $\sigma(f_2^{-1})=f_2$, we have $x_4=h_2^{-1}$, and so $x_5=e_1$. Therefore, the equivalence class containing $e_1$ is $v_1=\{x_1, x_2, x_3, x_4\}= \{e_1, h_3,f_2^{-1},h_2^{-1}\}.$ Similarly, the other equivalence classes are given by
$v_2= \{ e_1^{-1},g_1,f_2,g_3^{-1}\},\,v_3=\{e_2,g_2,f_1^{-1},g_1^{-1}\},\,v_4=\{e_2^{-1},h_1,f_1,h_3^{-1}\},
v_5=\{e_3,h_2,f_3,h_1^{-1}\},\,v_6=\{e_3^{-1},g_3,f_3^{-1},g_2^{-1}\}$. Thus we have $V=E/\!\!\sim = \{v_1, v_2.\ldots,v_6\}$.
\item $\displaystyle \sigma_0=\prod_{i=1}^6 \tau_i$, where $\tau_1=(e_1, h_3,f_2^{-1},h_2^{-1}),\,\tau_2=( e_1^{-1},g_1,f_2,g_3^{-1}),\,\tau_3=(e_2,g_2,f_1^{-1},g_1^{-1}),\,\tau_4=(e_2^{-1},h_1,f_1,h_3^{-1}),\,\tau_5=(e_3,h_2,f_3,h_1^{-1}),\,\tau_6=(e_3^{-1},g_3,f_3^{-1},g_2^{-1})$.
\end{enumerate}

Define $F:E \to E$ induced by the action of $D_i$ on ${\tilde \p_i}$ by 
\begin{gather*}
 \displaystyle \omega_F=\prod_{i=1}^4\omega_i, \text{ where }\omega_1=(e_1,e_2,e_3,e_1^{-1},e_2^{-1},e_3^{-1}), \,\omega_2=(g_1,h_1,g_3,h_3,g_2,h_2),\,\\ \omega_3=(f_1,f_3^{-1},f_2^{-1},f_1^{-1},f_3,f_2),\,\omega_4=(g_1^{-1},h_1^{-1},g_3^{-1},h_3^{-1},g_2^{-1},h_2^{-1}).
\end{gather*}
 A simple calculations shows that the action of $F$ on $V$ is given by $$(v_1,v_3,v_5)(v_2,v_4,v_6),$$ and it is easy to see that $F\circ \sigma_1=\sigma_1\circ F$ and  $F\circ \sigma_0=\sigma_0\circ F$. Therefore, $F \in \text{Aut}(\Gamma)$.
\end{example}

\noindent In the following remark, we generalize the idea in Example~\ref{eg:fg_cpair}, to describe a construction of a fat graph automorphism that corresponds to an $(r,s)$-compatible pair of irreducible Type 1 actions.

\begin{remark}\label{rem:comp_fat_graphs}
Consider a compatible pair $(F_1,F_2)$ of irreducible fat graph automorphisms as in Definition~\ref{def:fat_graph_pair}, where the $F_i \in \text{Aut}(\Gamma_i)$, and let $V_i$ denote the vertex set of $\Gamma_i$. Suppose that $|V_i / \langle F_i \rangle | = 2$. Then $F_i \vert_{V_i}=(v_1^i,\dots,v_{k_1^i}^i)(w_1^i,\dots,w_{k_2^i}^i)\text{ for }i=1,2, \text{ where }O_1= (v_1^1,\dots,v_{k_1^1}^1)\text{ and} $
$O_2=(v_1^2,\dots,v_{k_1^2}^2)$ are compatible (in the sense that they correspond to orbits that were involved in the compatibility of the $D_i$).  Let $v_i^j=(e_{i1}^j,\dots,e_{i n_1^j}^j)$ for $1\leq i\leq k_1^1 \text{ and } j=1,2,$ with $k_i^j=n/n_i^j$. Observe that when we remove invariant (pairwise disjoint) disks around vertices in the compatible orbits, the local picture at each vertex $v_i^j$ in $O_j$   is transformed to $(e_{i1}^j,g_{i1}^j,e_{i2}^j,g_{i2}^j,\dots, e_{in_1^1}^j,g_{in_1^1}^j),$ where the boundary circles are the concatenations of the arcs $g_{ik}^j$. We now define a fat graph $\lp \Gamma_1, \Gamma_2\rp :=(E,\sim,\sigma_0,\sigma_1)$ as follows:
\begin{enumerate}[(i)]
\item $E=\{ e_{ij}^1,(e_{ij}^1)^{-1},g_{ij},g_{ij}^{-1},e_{ij}^2,(e_{ij}^2)^{-1},|\text{ for }\,1\leq i \leq n/n_1^1,1\leq j\leq n_1^1\}$. 
\item The vertices of $\Gamma$ are given by the following procedure:  Let $x_1=e_{ij}^1$, take $x_2=g_{ij}$. There exist $1\leq t_{ij}\leq n-1$ such that $e_{ij}^1=F_1^{t_{ij}}(e_{11}^1)$.  If $e=F_2^{t_{ij}}(e_{11}^2)$ and $v$ is the unique vertex containing $e$, then we define $x_3$ to be the succeeding edge to $e$ in the cyclic order at $v$ in $\Gamma_2$ and so on. Continuing this way, we obtain the set of all vertices and hence the fat graph $\lp\Gamma_1, \Gamma_2\rp$. 
\end{enumerate} 
Now, we define a map  $F=\lp F_1, F_2\rp$ on $\lp \Gamma_1, \Gamma_2\rp$ as follows: 
$$ F(e)= \begin{cases}
F_1(e), & \text { if } e=e_{ij}^1\text{ or }(e_{ij}^1)^{-1}, \text{ and } \\
F_2(e), & \text { if } e=e_{ij}^2\text{ or } (e_{ij}^2)^{-1}, \\
g_{kl}, & \text { if } e=g_{ij} \text{ and } F_1(e_{ij}^1)=e_{kl}, \text{ and } \\
g_{kl}^{-1}, & \text { if } e=g_{ij} \text{ and } F_1(e_{ij}^2)=e_{kl}.
\end{cases}$$
  It is straightforward to see that $F$ is an automorphism of $\Gamma$. An analogous construction works for the case when $|V_i / \langle F_i \rangle | = 1$.
 \end{remark}
  \noindent The discussion in Remark~\ref{rem:comp_fat_graphs} leads us to the following result.
   
\begin{theorem}
\label{thm:comp_fat_graphs}
There is a bijective correspondence between $(r,s)$-compatible pairs of irreducible Type 1 actions and compatible pairs of irreducible fat graph automorphisms via the correspondence $\lp D_1,D_2 \rp \leftrightarrow \lp F_{D_1}, F_{D_2} \rp$.
 \end{theorem}
 
 \noindent By an inductive application of Theorems~\ref{thm:comp_fat_graphs} and~\ref{thm:arb_real}, we can obtain the following result. 
 
 \begin{theorem}
 Every cyclic action that decomposes into compatibilities between finitely many irreducible Type 1 actions determines a fat graph automorphism. 
 \end{theorem}
 
\section{Some applications of the geometric realizations}\label{sec:2}
In this section, we derive some applications of the theory we have developed in Section~\ref{sec:1}. 
\subsection{Maximal reduction systems} 
In this subsection, we will apply our realizations to derive the size of maximal reduction systems of Type 1 actions and Type 2 compatible pairs, which are the essential building blocks of arbitrary non-rotational cyclic actions. 

\begin{theorem}\label{thm:size_max_red_sys}
Let $h$ be a cyclic action on $S_g$.
\begin{enumerate}[(i)]
\item If $D$ is of Type 1, then there exists a maximal reduction system $\C$ for $D$ such that
$$|\C|  = 
\begin{cases}
n(3g_0(D)-1), & \text{if } g_0(D)>1, \text{ and} \\
2n, & \text{if } g_0(D)=1.
\end{cases}$$
\item If $D$ is a Type 2 actions realizable as a compatible pair $(D_1,D_2)$ of Type 1 actions $D_i$ on $S_{g_i}$, then there exists a maximal reduction system $\C$ for $D$ such that
$$\small |\C| = \begin{cases}
n(3g_0(D_1)+3g_0(D_2)-2)+k, & \text { if }g_0(D_1),g_0(D_2)>1 ,  \\
n(3g_0(D_1)-1)+k+2n, &\text{ if }g_0(D_1)>1,g_0(D_2)=1,\\
n(3g_0(D_2)-1)+k+2n, &\text{ if }g_0(D_2)>1,g_0(D_1)=1,\\
n(3g_0(D_2)-1)+k, &\text{ if }g_0(D_2)>1,g_0(D_1)=0,\\
n(3g_0(D_1)-1)+k, &\text{ if }g_0(D_1)>1,g_0(D_2)=0,\text{ and}\\
k, & \text{otherwise,}
\end{cases}$$
where $k = g -g_1-g_2+1$.
\end{enumerate}
\end{theorem}

\begin{proof}
It follows from Theorem~\ref{lem:sph_act_to_g} and Remark~\ref{rem:triv_self_comp} that $\tilde D$ can be realized as $\l D,g\r$ type, where $D$ is a Type 1 action with $g_0(D)=0$. Part (i) now follows from the fact that $D$ is irreducible~\cite{JG3}, and the fact that a pants decomposition for $S_{g_0(\tilde D),1}$ comprises $3g_0(\tilde D)-2$ nonisotopic curves, when $g_0(\tilde D)>1$. Finally, (ii) is a direct application of (i) and  Remark~\ref{rem:geom_comp_pair}.
\end{proof}

\noindent We conclude this subsection by pointing out that by an inductive application of Theorem~\ref{thm:arb_real}, one can generalize Theorem~\ref{thm:size_max_red_sys} to arbitrary Type 2 actions. 

\subsection{Symplectic representations of cyclic actions}
 In this subsection, we will apply the realizations obtained in Section~\ref{sec:1} to describe a procedure for deriving the image of a non-rotational cyclic action under the symplectic representation
$\Psi: \text{Mod}(S_g) \to \text{Sp}(2g; \mathbb{Z}).$ We will later extend these results to obtain the representations of the roots of Dehn twists.

Consider a polygon $\p = \p_{2k}$ with a reduced boundary word $W = W(\p)$, when read in the counter-clockwise sense. A direct application of the methods detailed in~\cite[Chapter 3]{GR} shows that $W(\p)$ is of the form $$Q a R b S a^{-1} T b^{-1} U,$$ for some words $Q, R, S, T,U$ (possibly empty), and letters $a,b$. Then we have the following proposition. 

\begin{proposition}\label{prop:normal}
Let $W(\p) =QaRbSa^{-1}Tb^{-1}U$. Suppose that $\p'$ is the polygon with $W(\p') = QTSRUxyx^{-1}y^{-1}$ obtained by applying the handle normalization algorithm once to $\p$. Then $x$ and $y$ are homotopically equivalent to $QTb^{-1}U$ and $U^{-1}R^{-1}a^{-1}Tb^{-1}U$, respectively. 
\end{proposition}
\begin{proof}
In the following three steps, we obtain the reduced word $W'$ equivalent to $QTSRUxyx^{-1}y^{-1}$. Furthermore, we would like to express the new letters $x,y$ in terms of the letters in the word $W$. 

\noindent \textbf{Step 1.} Let $x$ denote the diagonal of the polygon $\p$ joining the end point of $a$ and the common vertex of $Q$ and $U$.  We cut $QaRbSa^{-1}Tb^{-1}U$ along the diagonal $x$  and obtain two polygons $QaRbSx$ and $x^{-1}a^{-1}Tb^{-1}U$ (see Figure~\ref{fig:4}). Next, we glue the polygons along $a$ to get $x^{-1} R b S x Q T b^{-1} U$ (see Figure~\ref{fig:5}). 

\tikzset{->-/.style={decoration={
  markings,
  mark=at position .5 with {\arrow{>}}},postaction={decorate}}}
 
\begin{figure}[H]
\begin{center}
\begin{tikzpicture}
\begin{scope}[every node/.style={sloped,allow upside down}]
\draw [dotted] (2, 0) -- (1.61, 1.17); \draw [->-] (0.62, 1.9)--(1.61, 1.17); \draw [dotted] (0.62,1.9)--(-0.62, 1.9); \draw [->-] (-1.61, 1.17) -- (-0.62, 1.9); \draw [dotted] (-1.41, -1.41) -- (-2, 0) -- (-1.61, 1.17); \draw [->-] (-1.41, -1.41) -- (0, -2); \draw [dotted] (0, -2) -- (1.41, -1.41); \draw [->-] (1.41, -1.41) -- (2, 0); \draw [->-, red] (2, 0) -- (-2, 0);

\draw [fill,red] (-2,0) circle [radius=0.03]; \draw [fill,red] (2,0) circle [radius=0.03]; \draw [fill,red] (1.61, 1.17) circle [radius=0.03]; \draw [fill,red] (-1.61, 1.17) circle [radius=0.03]; \draw [fill,red] (0.62,1.9) circle [radius=0.03]; \draw [fill,red] (-0.62,1.9) circle [radius=0.03]; \draw [fill,red] (1.41, -1.41) circle [radius=0.03]; \draw [fill,red] (0, -2) circle [radius=0.03]; \draw [fill,red] (-1.41, -1.41) circle [radius=0.03];

\draw (-2, 0.5) node {$Q$}; \draw (2, 0.5) node {$S$}; \draw (1.25, 1.65) node {$b$}; \draw (-1.25, 1.65) node {$a$}; \draw (0,2.1) node {$R$}; \draw (0, 0) node [above] {$x$};\draw (1.85, -0.8) node {$a$}; \draw (-1.9, -0.8) node {$U$}; \draw (0.8, -1.9) node {$T$}; \draw (-0.8, -1.9) node {$b$}; 

\draw [dotted] (3, 0) -- (4, 1);  \draw [->-]  (5.5,0.75) -- (4, 1); \draw [->-] (5.5,-0.75) -- (5.5, 0.75); \draw  [dotted] (4, -1) -- (5.5,-0.75); \draw [->-] (3, 0) -- (4, -1); 

\draw [fill,red] (3,0) circle [radius=0.03]; \draw [fill,red] (4,1) circle [radius=0.03];\draw [fill,red] (5.5,0.75) circle [radius=0.03]; \draw [fill,red] (5.5,-0.75) circle [radius=0.03]; \draw [fill,red] (4,-1) circle [radius=0.03];

\draw (3.3, 0.6) node {$U$}; \draw (4.7, 1.1) node {$x$};\draw (5.5, 0) node [left] {$a$};  \draw (4.7, -1.1) node {$T$}; \draw (3.3, -0.6) node {$b$};

\draw [->-] (6,-0.75) -- (6, 0.75); \draw [dotted] (6,0.75) -- (7,2); \draw [->-] (7,2)--(8,0.75); \draw [dotted] (8,0.75) -- (8, -0.75); \draw [->-] (8, -0.75) -- (7, -2); \draw [dotted] (7, -2) -- (6, -0.75); 

\draw [fill,red] (6,-0.75) circle [radius=0.03]; \draw [fill,red] (6,0.75) circle [radius=0.03];\draw [fill,red] (8,0.75) circle [radius=0.03]; \draw [fill,red] (8,-0.75) circle [radius=0.03]; \draw [fill,red] (7,2) circle [radius=0.03]; \draw [fill,red] (7,-2) circle [radius=0.03];

\draw (6, 0) node [right] {$a$}; \draw (8, 0) node [right] {$S$}; \draw (6.3, 1.4) node {$R$}; \draw (7.7, 1.4) node {$b$}; \draw (6.3, -1.4) node {$Q$}; \draw (7.7, -1.4) node {$x$};
\draw [->, green] (2.2,0) -- (2.7,0);
\end{scope}
\end{tikzpicture}
\caption{Cut the polygon $QaRbSa^{-1}Tb^{-1}U$ (left) along the diagonal $x$ and get two polygons $RbSxQa$ (right) and $a^{-1}Tb^{-1}Ux^{-1}$ (middle).} 
\label{fig:4}
\end{center}
\end{figure}
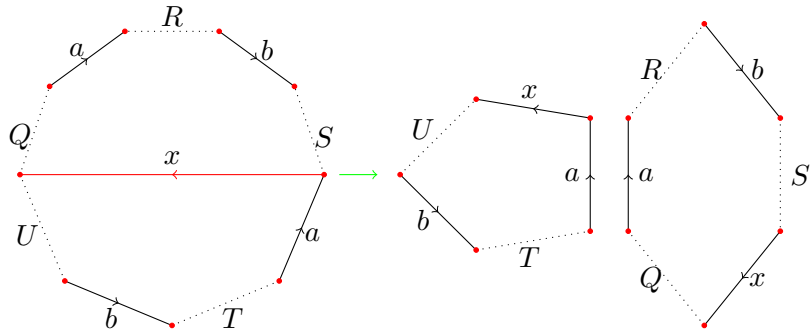

\noindent In the new polygon, the image of $a$ is described in Figure~\ref{fig:5}.
\begin{figure}[H]
\begin{center}
\begin{tikzpicture}
\begin{scope}[every node/.style={sloped,allow upside down}]
\draw [dotted] (2, 0) -- (1.61, 1.17); \draw [->-] (0.62, 1.9)--(1.61, 1.17); \draw [dotted] (0.62,1.9)--(-0.62, 1.9); \draw [->-] (-0.62, 1.9)--(-1.61, 1.17); \draw [->-] (-2, 0) -- (-1.41, -1.41); \draw [dotted] (-2, 0) -- (-1.61, 1.17); \draw [dotted] (-1.41, -1.41) -- (0, -2); \draw [dotted] (0, -2) -- (1.41, -1.41); \draw [->-] (2,0) -- (1.41, -1.41); \draw [->-, green] (0,-2) -- (-0.62, 1.9);

\draw [fill,red] (-2,0) circle [radius=0.03]; \draw [fill,red] (2,0) circle [radius=0.03]; \draw [fill,red] (1.61, 1.17) circle [radius=0.03]; \draw [fill,red] (-1.61, 1.17) circle [radius=0.03]; \draw [fill,red] (0.62,1.9) circle [radius=0.03]; \draw [fill,red] (-0.62,1.9) circle [radius=0.03]; \draw [fill,red] (1.41, -1.41) circle [radius=0.03]; \draw [fill,red] (0, -2) circle [radius=0.03]; \draw [fill,red] (-1.41, -1.41) circle [radius=0.03];

\draw (-2, 0.5) node {$U$}; \draw (2, 0.5) node {$S$}; \draw (1.25, 1.65) node {$b$}; \draw (-1.25, 1.65) node {$x$}; \draw (0,2.1) node {$R$}; \draw [green] (0, 0) node {$a$};\draw (1.85, -0.8) node {$x$}; \draw (-1.9, -0.8) node {$b$}; \draw (0.8, -1.9) node {$Q$}; \draw (-0.8, -1.9) node {$T$}; 
\end{scope}
\end{tikzpicture} 
\caption{The polygon obtained by attaching the two polygons in Figure~\ref{fig:4} (right) along the sides labelled by $a$.}
\label{fig:5}
\end{center}
\end{figure}
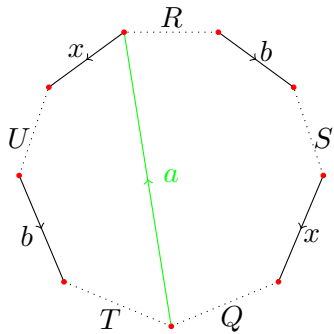

\noindent \textbf{Step 2.} In this step, we cut the polygon obtained at the end of Step 1 along the diagonal $y$ (see Figure~\ref{fig:6}, left) and then paste along $b$ to obtain $x^{-1}RUy^{-1}QTSxy$ (Figure~\ref{fig:7}). 
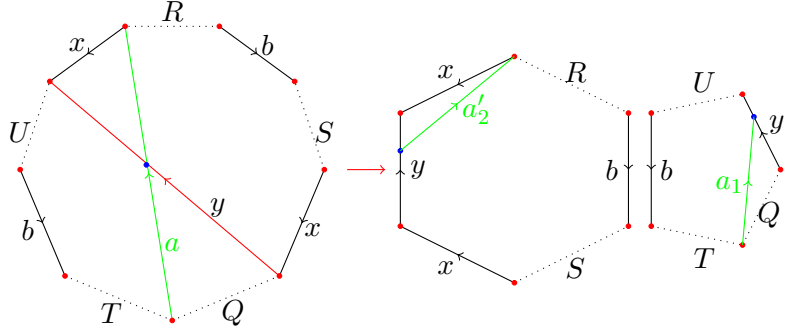
\begin{figure}[H]
\begin{center}
\begin{tikzpicture}
\begin{scope}[every node/.style={sloped,allow upside down}]
\draw [dotted] (2, 0) -- (1.61, 1.17); \draw [->-] (0.62, 1.9)--(1.61, 1.17); \draw [dotted] (0.62,1.9)--(-0.62, 1.9); \draw [->-] (-0.62, 1.9)--(-1.61, 1.17); \draw [->-] (-2, 0) -- (-1.41, -1.41); \draw [dotted] (-2, 0) -- (-1.61, 1.17); \draw [dotted] (-1.41, -1.41) -- (0, -2); \draw [dotted] (0, -2) -- (1.41, -1.41); \draw [->-] (2,0) -- (1.41, -1.41); \draw [->-, green] (0,-2) -- (-0.62, 1.9);
\draw [->-, red] (1.41, -1.41) -- (-1.61, 1.17);\draw [fill,blue] (-0.34, 0.063) circle [radius=0.03];

\draw [fill,red] (-2,0) circle [radius=0.03]; \draw [fill,red] (2,0) circle [radius=0.03]; \draw [fill,red] (1.61, 1.17) circle [radius=0.03]; \draw [fill,red] (-1.61, 1.17) circle [radius=0.03]; \draw [fill,red] (0.62,1.9) circle [radius=0.03]; \draw [fill,red] (-0.62,1.9) circle [radius=0.03]; \draw [fill,red] (1.41, -1.41) circle [radius=0.03]; \draw [fill,red] (0, -2) circle [radius=0.03]; \draw [fill,red] (-1.41, -1.41) circle [radius=0.03];

\draw (-2, 0.5) node {$U$}; \draw (2, 0.5) node {$S$}; \draw (1.25, 1.65) node {$b$}; \draw (-1.25, 1.65) node {$x$}; \draw (0,2.1) node {$R$}; \draw [green] (0, -1) node {$a$};\draw (1.85, -0.8) node {$x$}; \draw (-1.9, -0.8) node {$b$}; \draw (0.8, -1.9) node {$Q$}; \draw (-0.8, -1.9) node {$T$}; \draw (0.6, -0.5) node {$y$}; 
\draw [->, red] (2.3, 0) -- (2.8,0);

\draw [->-] (3, -0.75)--(3, 0.75); \draw [->-] (6, 0.75) -- (6,-0.75); \draw [dotted] (4.5,1.5) -- (6, 0.75);  \draw [->-] (4.5,1.5) -- (3, 0.75); \draw [dotted] (4.5,-1.5) -- (6, -0.75);  \draw [->-] (4.5,-1.5) -- (3, -0.75); 
\draw [->-](6.3, 0.75) -- (6.3, -0.75); \draw [dotted] (6.3, 0.75)--(7.5,1); \draw [dotted] (6.3, -0.75)--(7.5,-1); \draw [->-] (8,0) -- (7.5, 1); \draw [dotted] (8,0) -- (7.5, -1); 

\draw (3, 0) node [right] {$y$}; \draw (3.6, -1.3) node {$x$};\draw (3.6, 1.3) node {$x$}; \draw (5.3, -1.3) node {$S$};\draw (5.3, 1.3) node {$R$}; \draw (6, 0) node [left] {$b$}; \draw (6.5, 0) node {$b$}; \draw (7, 0.9) node [above] {$U$}; \draw (7, -0.9) node[below] {$T$}; \draw (7.95, 0.6) node {$y$}; \draw (7.85, -0.6) node {$Q$};

\draw [fill,red] (3,0.75) circle [radius=0.03]; \draw [fill,red] (3,-0.75) circle [radius=0.03]; \draw [fill,red] (6, 0.75) circle [radius=0.03]; \draw [fill,red] (6, -0.75) circle [radius=0.03]; \draw [fill,red] (4.5,1.5) circle [radius=0.03]; \draw [fill,red] (4.5,-1.5) circle [radius=0.03]; \draw [fill,red] (6.3, 0.75) circle [radius=0.03]; \draw [fill,red] (6.3, -0.75) circle [radius=0.03]; \draw [fill,red] (7.5, 1) circle [radius=0.03]; \draw [fill,red] (7.5, -1) circle [radius=0.03]; \draw [fill,red] (8, 0) circle [radius=0.03];
\draw [fill,blue] (3, 0.25) circle [radius=0.03]; \draw [->-, green] (3, 0.25) -- (4.5, 1.5); \draw [green] (4, 0.8) node {$a'_2$};
\draw [fill,blue] (7.65, 0.7) circle [radius=0.03];
\draw [->-, green] (7.5, -1) -- (7.65, 0.7); \draw [green] (7.65, -0.2) node [left] {$a_1$};
\end{scope}

\end{tikzpicture}
\caption{Cutting the polygon in Figure~\ref{fig:5} along the diagonal $y$ yields two polygons (right).}
\label{fig:6}
\end{center}
\end{figure}

Note that, $y$ intersects $a$ at an interior point which divides $a$ into two segments $a_1, a'_2$. The images of $a_1, a'_2$ and $b$  are described in Figure~\ref{fig:7}. Note that $a=a_1*a_2'.$
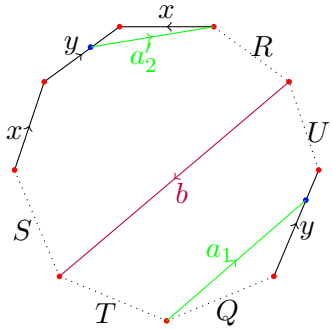
\begin{figure}[H]
\begin{center}
\begin{tikzpicture}
\begin{scope}[every node/.style={sloped,allow upside down}]
\draw [dotted] (2, 0) -- (1.61, 1.17); \draw [dotted] (0.62, 1.9)--(1.61, 1.17); \draw [->-] (0.62,1.9)--(-0.62, 1.9); \draw [->-] (-1.61, 1.17) -- (-0.62, 1.9); \draw [dotted] (-2, 0) -- (-1.41, -1.41); \draw [->-] (-2, 0) -- (-1.61, 1.17); \draw [dotted] (-1.41, -1.41) -- (0, -2); \draw [dotted] (0, -2) -- (1.41, -1.41); \draw [->-] (1.41, -1.41)--(2,0); 

\draw [fill,red] (-2,0) circle [radius=0.03]; \draw [fill,red] (2,0) circle [radius=0.03]; \draw [fill,red] (1.61, 1.17) circle [radius=0.03]; \draw [fill,red] (-1.61, 1.17) circle [radius=0.03]; \draw [fill,red] (0.62,1.9) circle [radius=0.03]; \draw [fill,red] (-0.62,1.9) circle [radius=0.03]; \draw [fill,red] (1.41, -1.41) circle [radius=0.03]; \draw [fill,red] (0, -2) circle [radius=0.03]; \draw [fill,red] (-1.41, -1.41) circle [radius=0.03];

\draw (-2, 0.5) node {$x$}; \draw (2, 0.5) node {$U$}; \draw (1.25, 1.65) node {$R$}; \draw (-1.25, 1.65) node {$y$}; \draw (0,2.1) node {$x$}; \draw (1.85, -0.8) node {$y$}; \draw (-1.9, -0.8) node {$S$}; \draw (0.8, -1.9) node {$Q$}; \draw (-0.8, -1.9) node {$T$}; 

\draw [->-, purple] (1.61, 1.17) -- (-1.41,-1.41); \draw [purple] (0.2, -0.3) node {$b$};
\draw [fill,blue] (1.83, -0.4) circle [radius=0.03]; \draw [->-, green] (0,-2) -- (1.83, -0.4); \draw [green] (0.7, -1.1) node {$a_1$}; \draw [fill,blue] (-1, 1.63) circle [radius=0.03]; \draw [->-, green] (-1, 1.63) -- (0.62,1.9); \draw [green] (-0.3, 1.5) node {$a'_2$};
\end{scope}

\end{tikzpicture} 
\caption{The polygon obtained by attaching the two polygons in Figure~\ref{fig:6} (right) along the sides labelled by $b$.}
\label{fig:7}
\end{center}
\end{figure}

\noindent \textbf{Step 3.} Finally, we cut the polygon $x^{-1} R U y^{-1} Q T S x y$ along the diagonal $z$ (see Figure~\ref{fig:8}, left) and obtain two polygons $x^{-1} R U z$ and $z^{-1} y^{-1} Q T S x y$. Next, we glue them along $x$ and obtain the polygon given by $Q T S R U zyz^{-1} y^{-1}$ (Figure~\ref{fig:9}).
\begin{figure}[H]
\begin{center}
\begin{tikzpicture}
\begin{scope}[every node/.style={sloped,allow upside down}]
\draw [dotted] (2, 0) -- (1.61, 1.17); \draw [dotted] (0.62, 1.9)--(1.61, 1.17); \draw [->-] (0.62,1.9)--(-0.62, 1.9); \draw [->-] (-1.61, 1.17) -- (-0.62, 1.9); \draw [dotted] (-2, 0) -- (-1.41, -1.41); \draw [->-] (-2, 0) -- (-1.61, 1.17); \draw [dotted] (-1.41, -1.41) -- (0, -2); \draw [dotted] (0, -2) -- (1.41, -1.41); \draw [->-] (1.41, -1.41)--(2,0); 

\draw [fill,red] (-2,0) circle [radius=0.03]; \draw [fill,red] (2,0) circle [radius=0.03]; \draw [fill,red] (1.61, 1.17) circle [radius=0.03]; \draw [fill,red] (-1.61, 1.17) circle [radius=0.03]; \draw [fill,red] (0.62,1.9) circle [radius=0.03]; \draw [fill,red] (-0.62,1.9) circle [radius=0.03]; \draw [fill,red] (1.41, -1.41) circle [radius=0.03]; \draw [fill,red] (0, -2) circle [radius=0.03]; \draw [fill,red] (-1.41, -1.41) circle [radius=0.03];
\draw [fill, brown] (-0.39,1.72) circle [radius=0.03];\draw [fill,black] (1.05,0.69) circle [radius=0.03];

\draw (-2, 0.5) node {$x$}; \draw (2, 0.5) node {$U$}; \draw (1.25, 1.65) node {$R$}; \draw (-1.25, 1.65) node {$y$}; \draw (0,2.1) node {$x$}; \draw (1.85, -0.8) node {$y$}; \draw (-1.9, -0.8) node {$S$}; \draw (0.8, -1.9) node {$Q$}; \draw (-0.8, -1.9) node {$T$}; 

\draw [->-, purple] (1.61, 1.17) -- (-1.41,-1.41); \draw [purple] (0.2, -0.3) node {$b$};
\draw [fill,blue] (1.83, -0.4) circle [radius=0.03]; \draw [->-, green] (0,-2) -- (1.83, -0.4); \draw [green] (0.7, -1.1) node {$a_1$}; \draw [fill,blue] (-1, 1.63) circle [radius=0.03]; \draw [->-, green] (-1, 1.63) -- (0.62,1.9); \draw [green] (0.2, 1.6) node {$a'_2$};

\draw [->-, red] (2,0) -- (-0.62, 1.9); \draw [red] (0.6, 0.8) node {$z$};
\draw [->, red] (2.2, -0.25) -- (2.6,-0.25);
\draw [dotted] (2.7,0.75) -- (2.7, -0.75) -- (4, -0.5); \draw [->-] (4, -0.5) -- (4, 0.5); \draw [->-] (2.7, 0.75) -- (4, 0.5); 

\draw [fill,red] (2.7, 0.75) circle [radius=0.03];\draw [fill,red] (2.7, -0.75) circle [radius=0.03];\draw [fill,red] (4, 0.5) circle [radius=0.03];\draw [fill,red] (4, -0.5) circle [radius=0.03];

\draw (2.7, 0) node [left] {$T$}; \draw (4, 0) node [right] {$x$}; \draw (3.7, 0.8) node [left] {$z$}; \draw (3.7, -0.85) node [left] {$Q$}; 

\draw [->-] (4.5, -0.5) -- (4.5, 0.5); \draw [->-] (4.5, 0.5) -- (6, 1); \draw [->-] (7.5, 0.5) -- (6,1); \draw [->-] (7.5, -0.5) -- (7.5, 0.5); \draw [dotted] (4.5, -0.5) -- (5.5, -1) -- (6.5, -1) -- (7.5, -0.5);
\draw [fill,red] (4.5, -0.5) circle [radius=0.03];\draw [fill,red] (4.5, 0.5) circle [radius=0.03];\draw [fill,red] (6, 1) circle [radius=0.03];\draw [fill,red] (7.5, 0.5) circle [radius=0.03]; \draw [fill,red] (7.5, -0.5) circle [radius=0.03];\draw [fill,red] (5.5, -1) circle [radius=0.03];\draw [fill,red] (6.5, -1) circle [radius=0.03];

\draw (4.5, 0) node [right] {$x$}; \draw (7.5, 0) node [right] {$y$}; \draw (5.4, 0.9) node [left] {$y$}; \draw (7, 0.9) node [left] {$z$}; \draw (6, -1.2) node {$T$}; \draw (5.1, -0.9) node [left] {$S$}; \draw (7.3, -0.95) node [left] {$Q$};

\draw [fill,blue] (7.5, 0.25) circle [radius=0.03]; \draw [->-, green] (6.5, -1) -- (7.5, 0.25); \draw [green] (7.2, -0.25) node [left] {$a_1$};

\draw [fill,blue] (7.1, 0.625) circle [radius=0.03]; \draw [->-,purple] (7.1, 0.625) -- (5.5, -1); \draw [purple] (6.3, -0.1) node [left] {$b_2$};

\draw [fill,blue] (5.4, 0.8) circle [radius=0.03]; \draw [fill,brown] (6.6, 0.8) circle [radius=0.03]; \draw [->-, green] (5.4, 0.8) -- (6.6, 0.8); \draw [green] (6,0.6) node {$a_2$};

\draw [fill,brown] (3.6, 0.57) circle [radius=0.03]; \draw [->-, green] (3.6, 0.57) -- (4, -0.5); \draw [green] (3.9, 0) node [left] {$a_3$};

\draw [fill, black] (3.15, 0.65) circle [radius=0.03]; \draw [->-, purple] (2.7, -0.75) -- (3.15, 0.65); \draw [purple] (2.9, 0) node [right] {$b_1$};
\end{scope}

\end{tikzpicture} 
\caption{Cut the polygon in Figure~\ref{fig:7} along the diagonal $z$ and get two polygons (right).}
\label{fig:8}
\end{center}
\end{figure}
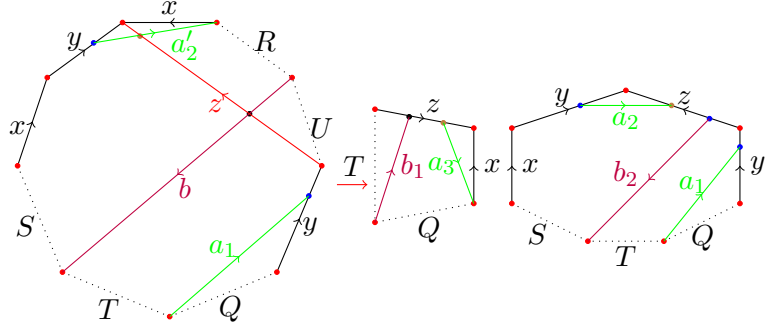

The diagonal $z$ intersects $a'_2$ at an interior point which divides $a_2'$ into segments $a_2, a_3$. Furthermore, $b$ splits into two segments $b_1$ and $b_2$.    
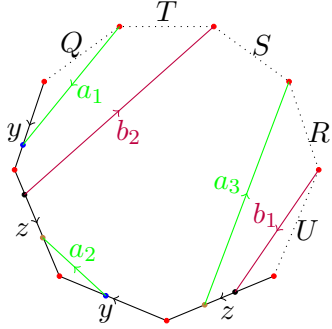
\begin{figure}[H]
\begin{center}
\begin{tikzpicture}
\begin{scope}[every node/.style={sloped,allow upside down}]
\draw [dotted] (2, 0) -- (1.61, 1.17); \draw [dotted] (0.62, 1.9)--(1.61, 1.17); \draw [dotted] (0.62,1.9)--(-0.62, 1.9); \draw [dotted] (-1.61, 1.17) -- (-0.62, 1.9); \draw [->-] (-2, 0) -- (-1.41, -1.41); \draw [->-] (-1.61, 1.17) -- (-2, 0); \draw [->-] (0, -2) -- (-1.41, -1.41); \draw [->-] (1.41, -1.41)-- (0, -2); \draw [dotted] (1.41, -1.41)--(2,0); 

\draw [fill,red] (-2,0) circle [radius=0.03]; \draw [fill,red] (2,0) circle [radius=0.03]; \draw [fill,red] (1.61, 1.17) circle [radius=0.03]; \draw [fill,red] (-1.61, 1.17) circle [radius=0.03]; \draw [fill,red] (0.62,1.9) circle [radius=0.03]; \draw [fill,red] (-0.62,1.9) circle [radius=0.03]; \draw [fill,red] (1.41, -1.41) circle [radius=0.03]; \draw [fill,red] (0, -2) circle [radius=0.03]; \draw [fill,red] (-1.41, -1.41) circle [radius=0.03];

\draw (-2, 0.5) node {$y$}; \draw (2, 0.5) node {$R$}; \draw (1.25, 1.65) node {$S$}; \draw (-1.25, 1.65) node {$Q$}; \draw (0,2.1) node {$T$}; \draw (1.85, -0.8) node {$U$}; \draw (-1.9, -0.8) node {$z$}; \draw (0.8, -1.9) node {$z$}; \draw (-0.8, -1.9) node {$y$}; 

\draw [fill,blue] (-1.89, 0.33) circle [radius=0.03]; \draw [->-, green] (-0.62, 1.9) -- (-1.89, 0.33); \draw [green] (-1, 1) node {$a_1$};  \draw [fill,brown] (-1.63, -0.9) circle [radius=0.03]; \draw [fill,blue] (-0.8, -1.67) circle [radius=0.03]; \draw [->-, green] (-0.8, -1.67) -- (-1.63, -0.9); \draw [green] (-1.1, -1.1) node {$a_2$}; \draw [fill,brown] (0.5, -1.79) circle [radius=0.03]; \draw [->-, green] (0.5, -1.79) -- (1.61, 1.17); \draw [green] (0.8, -0.2) node {$a_3$};

\draw [fill,black] (0.9, -1.62) circle [radius=0.03]; \draw [->-, purple] (2,0) -- (0.9, -1.62); \draw [purple] (1.3, -0.6) node {$b_1$};
\draw [fill,black] (-1.87, -0.33) circle [radius=0.03]; \draw [->-, purple] (-1.87, -0.33) -- (0.62, 1.9); \draw [purple] (-0.5, 0.5) node {$b_2$};
\end{scope}

\end{tikzpicture} 
\caption{The polygon obtained by attaching the two polygons in Figure~\ref{fig:8} (right) along the sides labelled by $z$.}
\label{fig:9}
\end{center}
\end{figure}
We can slide the final point of $b_1$ and the initial point of $b_2$ along the reverse direction of $z$ to the initial point of $z$ simultaneously by homotopy. Then $QTb_2^{-1} y^{-1}$ bound a disc in the surface and hence gives a word which is equivalent to identity. Thus we have 
\begin{eqnarray*}
y&=&QTb_2^{-1}=QTb_2^{-1}b_1^{-1}b_1 = QTb^{-1}b_1\\ \Rightarrow y&=&QTb^{-1}U
\end{eqnarray*}
Similarly, by homotopy, slide the final point of $a_1$ and the initial point of $a_2$ to the final point of $y$. Then we slide the final point of $a_2$ and the initial point in the reverse direction along $z$ to the initial point of $z$. The edges $z$ and $a_2$ bound a disc. Therefore we have, 
\begin{eqnarray*}
z&=&a_2^{-1}=a_3(a_3^{-1}a_2^{-1}a_1^{-1}) a_1=a_3a^{-1}a_1\\ &=& U^{-1}R^{-1} a^{-1}Tb^{-1}U.
\end{eqnarray*}
Therefore, by renaming the side $z$ by $x$, we have the proposition.
\end{proof}

\begin{notation}
Let $\p$ be a polygon with $S(\p) \approx S_g$. 
\begin{enumerate}[(a)]
\item We denote by $\N^i(\p)$, the polygon obtained from $\p$ after $i$ successive applications of the normalization procedure described in Proposition~\ref{prop:normal}.
\item We denote by $L(\p)$, the set of distinct letters in $W(\p)$. 
\item We denote by $\B(\p)$, the set of standard generators of $H_1(S_g;\mathbb{Z})$ expressed in terms of elements in $L(\p)$. 
\end{enumerate}
\end{notation} 

\noindent Let $W$ and $W'$ be as in Proposition~\ref{prop:normal}. Then the map $$\B(\p')\to \B(\p)\, :\, x \mapsto QTb^{-1}U, \, y \mapsto U^{-1}R^{-1}a^{-1}Tb^{-1}U,\, z \mapsto z,$$ for all $z\in \B(\p')\setminus \{x, y\}$, uniquely determines an isomorphism on $H_1(S_g;\mathbb{Z})$, which we denote by $f_{\p',\p}$. This brings us to the following lemma.

\begin{lemma}\label{lem:normalization}
Let $\p$ be a polygon with $S(\p) \approx S_g$. Then 
$$W(\N^g(\p))) = \displaystyle \prod_{i=1}^g [x_i,y_i],$$ and the mapping 
 $$\displaystyle f_\p= \prod_{i=1}^g f_{\N^i(\p),\,\N^{i-1}(\p)}$$ defines an isomorphism of the homology group $H_1(S_g;\mathbb{Z})$ such that $$\B(\p_g) \xmapsto{f_\p} \B(\p).$$ 
\end{lemma}

\noindent  For an isomorphism $\varphi: H_1(S_g;\mathbb{Z}) \to H_1(S_g;\mathbb{Z})$, we shall denote by $M_{\varphi}$, the matrix of $\varphi$ with respect to the standard homology generators. The following proposition, which is a direct consequence of Theorem~\ref{main} and Lemma~\ref{lem:normalization}, describes a procedure for finding the image of a Type 1 action under $\Psi$. 

\begin{theorem}\label{thm:rep type 2}
Consider the Type 1 action $$D = ((n,0;(c_1,n_1), (c_2,n_2), (c_3,n)).$$ Then  $\displaystyle \Psi(D) = M_{\varphi},  \text{ where } \varphi = f_{\p_D}^{-1}\phi_{\p_D}f_{\p_D},$ with $f_{\p_D}$ as in Lemma ~\ref{lem:normalization}, and $\B(\p_D) \xmapsto{\phi_{\p_D}} \B(\p_D)$ is induced by $(a_i,Q_i) \mapsto (a_j,Q_j),$ where
$$j \equiv \begin{cases}
i+2c_3^{-1} \pmod{2n} , & \text{if } n_1,n_2\neq 2, \text{ and}\\
i+c_3^{-1} \pmod{n},  & \text{otherwise}.
\end{cases}
$$
\end{theorem}
\noindent An immediate consequence of Theorem~\ref{thm:rep type 2} is the following corollary, which describes the symplectic representation of a Type 2 action realizable as an $(r,s)$-compatible pair across a pair of fixed points.

\begin{corollary}
Let $D$ be a Type 2 action that is realizable as an $(r,s)$-compatible pair $(D_1,D_2)$ of Type 1 actions  as in Lemma~\ref{lem:comp_pair} such that $g(D)=g(D_1)+g(D_2).$ Then 
\[
\Psi(D)=
\begin{bmatrix}
 \Psi(D_1) & 0 \\
 0 & \Psi(D_2)
\end{bmatrix},
\] where the blocks $\Psi(D_i)$ are obtained using Theorem~\ref{thm:rep type 2}. 
\end{corollary}

\noindent We motivate the symplectic representations of arbitrary $(r,s)$-compatibilities (between Type 1 actions) with this example. 

\begin{example}\label{eg:type-3}
Consider the $C_6$-action 
$$D=(6,0;(1,2),(1,2),(1,6),(5,6))$$ on $S_3$. It can be realized as the following $(2,2)$-compatible pair $(D_1,D_2)$ of Type 1 actions 
$$D_1=(6,0;(1,2),(1,3),(1,6))\text{ and }D_2= (6,0;(1,2),(2,3),(5,6)).$$  
\begin{figure}[H]
\labellist
\small
\pinlabel $l_1$ at 30 22
\pinlabel $m_1$ at 10 29
\pinlabel $m_2$ at 49 12
\pinlabel $\beta$ at 54 34
\pinlabel $\delta$ at 93 34
\pinlabel $\alpha$ at 74 47
\pinlabel $\gamma$ at 74 21
\pinlabel $l_2=\alpha\beta\gamma\delta$ at 74 34
\pinlabel $l_3$ at 120 22
\pinlabel $m_3$ at 138 30
\pinlabel $D_1\curvearrowright$ at 0 59
\pinlabel $\curvearrowleft D_2$ at 148 59
\endlabellist
\centering
\includegraphics[width = 50 ex]{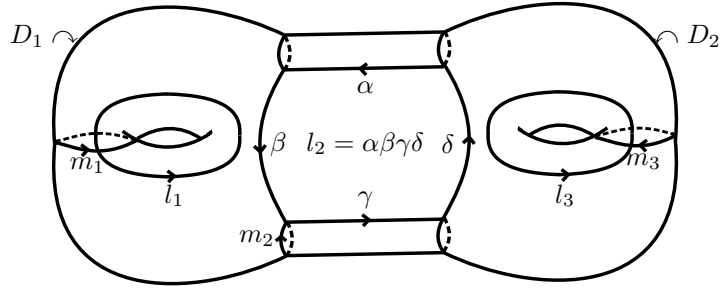}
\caption{A realization of the action $D$ on $S_3$.}
\label{fig:10}
\end{figure}

\noindent Then $\p_{D_i}$ is the hexagon with 
$$W(\p_{D_i})=a_{i,0}a_{i,1}a_{i,2}a_{i,0}^{-1}a_{i,1}^{-1}a_{i,2}^{-1}, \text{ for } i=1,2.$$  Let $\B(\p_D) = \{l_1,m_1,l_2,m_2,l_3,m_3\}$, as indicated in Figure~\ref{fig:10}. For $i=1,2$, $f_i=f_{\p_{D_i}}$ is given by:
$$\begin{array}{ll}
l_1 \xmapsto{f_1} a_{1,0}a_{1,1}, & m_1 \xmapsto{f_1} a_{1,2}a_{1,0}^{-1} \\
l_3 \xmapsto{f_2} a_{2,0}a_{2,1}, & m_3 \xmapsto{f_2} a_{2,2}a_{2,0}^{-1}.
\end{array}$$
The map 
$t_1=t_{\p_{D_1}}$ is induced by
$$
a_{1,0} \mapsto a_{1,1}, \, a_{1,1} \mapsto a_{1,2},\, a_{1,2} \mapsto a_{1,0}^{-1}, 
$$ 
while $t_2=t_{\p_{D_2}}$ is induced by
$$
a_{2,0} \mapsto a_{2,2}^{-1},\, a_{2,1} \mapsto a_{2,0},\, a_{2,2} \mapsto a_{2,1}.
$$
Then by Theorem~\ref{thm:rep type 2},  $f_i^{-1}t_i f_i$ gives the action of $D_i$  on $\B(\p_{D_i})$.
 By the realization of $D$, we have $$\Psi(D)(m_2)=-m_2.$$ To understand the action of $D$ on $l_2$, we write 
 $$l_2=\alpha\beta\gamma\delta, \text{ where } \beta \sim a_{1,0}^{-1}, \delta \sim a_{2,0}^{-1},$$ so that $$\alpha\beta\gamma\delta \to \gamma^{-1}a_{1,1}^{-1}\alpha^{-1}a_{2,2}$$ under the action of $D$ on $H_1(S_3;\mathbb{Z}).$ Consequently, $\Psi(D)$ is given by
$$\small \left[ 
  \begin{array}{rrrrrr}
   1&-1&-1&0&0&0 \\
   1&0&0&0&0&0\\
    0&0&-1&0&0&0\\
    0&0&0&-1&0&0\\
    0&0&0&0&0&1\\
    0&0&1&0&-1&1
  \end{array}\right].$$
\end{example}

\noindent This example motivates the following definition. 

\begin{definition} \label{def:homology_pair}
Let $D$ be a Type 2 action  that arises from a $(r,s)$-compatible pair $(D_1,D_2)$ of Type 1 actions as in Lemma~\ref{lem:comp_pair}. 
\begin{enumerate}[(i)]
\item We define
 $$\B(D) = \B(\p_{D_1}) \sqcup \B(D_1,D_2) \sqcup \B(\p_{D_2}), \text{ where }$$
\begin{gather*} 
\B(\p_{D_1}) = \{l_i,m_i\,|\, 1 \leq i \leq g(D_1)\}, \\
\B(\p_{D_2}) = \{l_i,m_i\,|\, g(D_1)+A(D_1,D_2) \leq i \leq g(D)\}, \text{ and} \\
\B(D_1,D_2) = \{l_i,m_i\,|\ g(D_1)+1 \leq i \leq g(D_1)+A(D_1,D_2)-1\}, 
\end{gather*} 
where $\B(D_1,D_2)$ denotes the set of standard homology generators for the $A(D_1,D_2)-1$ extra genera obtained by gluing annuli across a pair of compatible orbits of size $A(D_1,D_2)$. 
 
 \begin{figure}[H]
\labellist
\small
\pinlabel $\alpha_1$ at 85 55
\pinlabel $\alpha_2$ at 85 42
\pinlabel $\alpha_3$ at 85 29
\pinlabel $\alpha_p$ at 85 14
\pinlabel $m_1'$ at 68 55
\pinlabel $m_2'$ at 68 42
\pinlabel $m_3'$ at 68 29
\pinlabel $m_p'$ at 68 14
\pinlabel $\beta_1$ at 53 49
\pinlabel $\beta_2$ at 53 35
\pinlabel $\delta_1$ at 95 49
\pinlabel $\delta_2$ at 95 35
\endlabellist
\centering
\includegraphics[width = 55 ex]{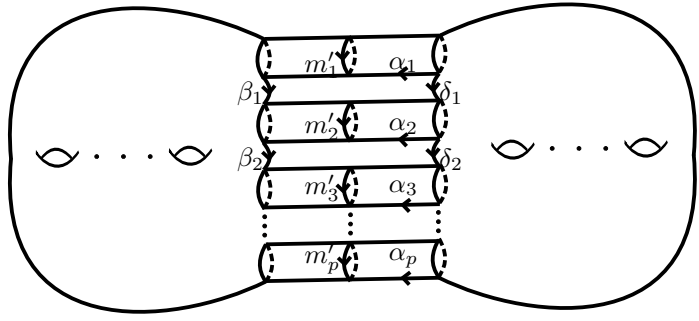}
\caption{A symplectic basis for the realization of an $(r,s)$-compatible pair $(D_1,D_2)$. Here, $p = A(D_1,D_2)$.}
\label{fig:11}
\end{figure}

\item Let $D,\B(D)$ be as in (i); let $O_1=\{v_1,\dots, v_{A(D_1,D_2)}\}$ and $O_2=\{w_1,\dots, w_{A(D_1,D_2)}\}$ correspond to compatible orbits in $\p_{D_1} \text{ and } \p_{D_2}$ respectively. Then we define $$l_{g(D_1)+i}= \alpha_i\beta_i\alpha_{i+1}^{-1}\delta_i^{-1}, \text{ for }1\leq i \leq A(D_1,D_2)-1,$$ and 
$$
m_{g(D_1)+i}= \begin{cases}
m_i' , & \text{if } i=1, \text{ and}\\
-m_{i+1}', &  \text{if } i=A(D_1,D_2)-1\\
m_i'+m_{g(D_1)+i-1}, & \text{otherwise},
\end{cases}$$ 
where the $\alpha_i$ and $m_i'$ are as indicated in Figure~\ref{fig:11}. 
\end{enumerate}
\end{definition}

\noindent Note that for $1\leq i \leq A(D_1,D_2)-1$, the arcs $\beta_i$ and $\delta_i$, as indicated in Figure~\ref{fig:11}, correspond to edge paths in $\p_{D_1}$ and $\p_{D_2}$, which are mapped under the actions of $D_1$ and $D_2$ to edge paths, which we denote by $E_{D_1}(\beta_i)$ and $E_{D_2}(\delta_i)$, respectively.  The following theorem, which gives an explicit description of the symplectic representations of Type 1 actions compatible pairs, is now a direct consequence of Theorem~\ref{thm:rep type 2}.

\begin{theorem}\label{thm:rep_comp_pair}
Let $D,\B(D)$ be as in Definition~\ref{def:homology_pair}.  Then the action of $D$ on $H_1(S_{g(D)};\mathbb{Z})$ is given as follows:
\begin{enumerate}[(i)]
\item for $i=1,2$,\,$\B(\p_{D_i})\xmapsto{\Psi(D)}\B(\p_{D_i})$ is given by Theorem~\ref{thm:rep type 2}.
\item for $l_j,m_j\in \B(D_1,D_2)$,  the action of $\Psi(D)$ is given by:
$$ 
\begin{array}{rcl}
 l_{j} & \mapsto & \alpha_{j+1}E_{D_1}(\beta_j)\alpha_j^{-1}E_{D_2}(\delta_j^{-1}), \text { and } \\ \\
m_{j} & \mapsto &
\begin{cases}
\displaystyle m_{j+1}-m_1,& \text{if }  j\neq g(D_1)+A(D_1,D_2)-1, \text{ and}  \\
-m_1,&\text{otherwise.}
\end{cases}
\end{array}
$$

\end{enumerate}
\end{theorem}
\begin{proof}
We can observe from Figure~\ref{fig:11} that 
$$
(\alpha_i,m_i')\mapsto\begin{cases}
\displaystyle (\alpha_1,m_1'), & \text{if } (i,j)=(p,p-1),\\
(\alpha_{i+1},m_{i+1}') , & \text{otherwise},
\end{cases}$$ with $p=A(D_1,D_2)$, which yields the result.
\end{proof}

\noindent Once again, by repeatedly applying Theorems~\ref{thm:rep_comp_pair} and~\ref{thm:arb_real}, one can obtain the symplectic representations of arbitrary Type 2 actions. However, for brevity, we shall refrain from explicitly stating any results to this effect.

\subsection{Symplectic representations of roots of Dehn twists}
Let $C$ be a nonseparating curve in $S_g$ for $g \geq 2$, and let $t_C \in \Mod(S_g)$ denote the left-handed Dehn twist about $C$. It is known~\cite{km,MS} that the conjugacy class of a root of $t_C$ of degree $n$ corresponds to the conjugacy class of a $C_n$-action on $S_{g-1}$ having two distinguished fixed points whose local rotational angles add up to $2\pi/n$ modulo $2 \pi$.  If such a $C_n$-action exists on $S_{g-1}$, then we remove  invariant discs around these fixed points, and attach a $1$-handle $A$ with a $(1/n)^{th}$ twist across the resultant boundary circles, thereby realizing a degree $n$ root of $t_C$ about the nonseparating curve $C$ in $A$. Consequently, the conjugacy classes of roots of $t_C$ of degree $n$ correspond to actions of the form
\begin{gather}
D= (n, g_0; (c_1,n_1),\ldots,(c_{\ell},n_{\ell})),  \text{ where }  \\ n_{\ell-1}=n_{\ell}=n \text{ and } c_{\ell-1}+c_{\ell} \equiv c_{\ell}c_{\ell-1} \pmod{n }. \nonumber \label{root_real_ds}
\end{gather}
 We call such an action a \textit{root realizing action.}
\noindent We denote the root of $t_C$ of degree $n$ corresponding to a root realizing action $D$ (as in ~\ref{root_real_ds}) by $h_D$. 

In view of the fact that the degree of $h_D$ is odd, and is bounded above by $2g(D)-1$~\cite[Corollary 2.2]{km}, the following result is a direct consequence of Lemma~\ref{lem:self_comp_ds}. 

\begin{proposition} \label{prop:data set_roots}
For $\ell \geq 4$, let $D=\{n, g_0; (c_1,n_1),\ldots,(c_{\ell},n_{\ell})\}$ be a root realizing Type 2 action, as in~\ref{root_real_ds}. Then $D$ can be realized as an $(\ell-1,3)$-compatible pair $(D_1,D_2),$ where
$$\begin{array}{lc}
D_1=(n, g_0; (c_1,n_1),\ldots,(c_{\ell-2},n_{\ell-2}),(c_{\ell-1}c_{\ell},n)), \text{ and } & \\
D_2=(n, 0; (c,n),(c_{\ell-1},n_{\ell-1}),(c_{\ell},n_{\ell})), &
\end{array}$$ satisfying $c \equiv -c_{\ell-1}c_{\ell} \pmod{n}.$
\end{proposition}

\noindent Note that $D_2$ in Proposition~\ref{prop:data set_roots} is in fact a root realizing Type 1 action. Hence, the problem of understanding the symplectic representation of an arbitrary root of $t_C$ reduces to understanding $\Psi(h_D)$, where $D$ is a root realizing Type 1 action. Let $h_D$ be (topologically) realized by removing invariant discs around the two distinguished fixed points $P$ and $Q$ of the action of $D$ on $S_g$, and then attaching a $1$-handle $A$ with a $(1/n)^{th}$ twist across the resultant boundary circles. Let $\alpha'$ be a path in $S_g$ from $P$ to $Q$ whose homotopy class (rel $\{P,Q\}$) does not intersect the free homotopy class of any essential simple closed curve in $\B(\p_D)$. Now choose an edge path $\alpha$ of the polygon $\p_D$ in the homotopy class of $\alpha'$ (rel $\{P,Q\})$. This brings us to the following theorem, which gives an explicit description of $\Psi(h_D)$. 

\begin{theorem}\label{thm:rep_roots}
Let 
 $$D=(2g+1,0; (a,2g+1), (b,2g+1), (c,2g+1)), $$ be a root realizing Type 1 action. Then $$\Psi(h_D) = \bigl(\begin{smallmatrix} E & B \\ C & I_2
\end{smallmatrix}\bigr),$$ where $E = \Psi(D)$ is a $2g\times 2g$ submatrix, $B$ is a $g\times 2$ submatrix, whose second column is zero and first column is determined by
$$\displaystyle f_{\p_D}^{-1}(\gamma \alpha^{-1}), \text{ where } \gamma = \phi_{\p_D}(\alpha),$$ and  $f_{\p_D} \text{ and } \phi_{\p_D}$ are as in Theorem~\ref{thm:rep type 2}, $I_2$ is the order $2$ identity matrix, and $C$ is uniquely determined by $A$ and $B$.
\end{theorem}

\begin{proof}
Let $\beta$ be an arc on $A$ joining terminal point of $\alpha$ to the initial point of $\alpha$ so that $\alpha\beta \simeq l_{g+1}$. Then $\alpha\beta \to \phi_{\p_D}(\alpha)\beta$ under $\Psi(h_D)$, and hence we have, $$\Psi(h_D)(l_{g+1})=f_{\p_D}^{-1}(\gamma\alpha^{-1})+l_{g+1},\text{ where }\gamma=\phi_{\p_D}(\alpha). $$ Also, by construction $\Psi(h_D)(m_{g+1})=m_{g+1}$. From this it follows that $\Psi(h_D)$ has the desired form. Since $\Psi(h_D)$ preserves the symplectic form, $C$ is uniquely determined by $E$ and $B$, and this completes the proof.   
\end{proof}

\noindent Let $C$ be a separating curve in $S_g$ so that $S_g = S_{g_1} \#_C S_{g_2}$. Then it is known~\cite{KR} that a root $h$ of $t_C$ of degree $n$ corresponds to a pair $(D_1,D_2)$ of cyclic actions $D_i$ on the $S_{g_i}$, for $i = 1,2$, having distinguished fixed points $P_i$ such that the local rotational angles $\theta_i$ induced around the $P_i$ (by the $D_i$) satisfy $$\theta_1 + \theta_2 \equiv 2\pi/n \pmod{2\pi}, \text{ where } n =\text{lcm}(n(D_1),n(D_2)).$$ We will call such a pair of actions a \textit{root realizing action pair of degree n.} Since the induced action of $t_C$ on $H_1(S_g;\mathbb{Z})$ is trivial, it follows immediately that:
\begin{theorem}
\label{thm:sep_curve_root_hom}
Let $C$ be a separating curve in $S_g$, and let $h$ be a root of $t_C$ of degree $n$ that corresponds to a root realizing action pair $(D_1,D_2)$ of degree $n$. Then 
$$\Psi(h) = \begin{bmatrix} 
                    \Psi(D_1) & 0 \\
                    0 & \Psi(D_2)
                    \end{bmatrix}.$$
\end{theorem} 

\noindent We conclude this subsection with the following interesting example.
\begin{example}\label{eg:root_rep}
 Consider the \textit{root realizing data set}
$$D=(5,0;(1,5),(2,5),(3,5),(4,5)).$$ 
It  can be realized by the following $(3,3)$-compatible type-2 data sets
$$ D_1=(5,0;(1,5),(2,5),(2,5))\text{ and }D_2=(5,0;(3,5),(3,5),(4,5)),$$
where $D_2$ is a \textit{root realizing data set}. 

 Let $\{l_5, m_5\}$ be the standard homology generator for the extra genus obtained by attaching a 1-handle with a  $1/5^{th}$ twist, across the boundary circles obtained by removing invariant discs around fixed points with local rotational angles $4\pi/5, 8\pi/5$, so that $\B(\p_D)\cup \{l_5 , m_5\}$ are standard homology generators for $S_5$. 
We have $$W(\p_{D_i})=a_{i,0}a_{i,1}\ldots a_{i,9},\, a_{i,2m+1}^{-1}\sim a_{i,2m+8},~1\leq m\leq 4,\, i=1,2.$$
For $i=1,2,\, 
f_i=f_{\p_{D_i}}$ is given  by:
$$
f_1(l_1)=a_{1,8}a_{1,6}^{-1}, f_1(m_1)=a_{1,4}^{-1}a_{1,8}, f_1(l_2)=a_{1,2}^{-1}a_{1,8}, f_1(m_2)=a_{1,8}a_{1,0}^{-1}a_{1,4}a_{1,2}^{-1}$$ and
$$f_2(l_3)=a_{2,8}a_{2,6}^{-1}, f_2(m_3)=a_{2,4}^{-1}a_{2,8}, f_2(l_4)=a_{2,2}^{-1}a_{2,8}, f_2(m_4)=a_{2,8}a_{2,0}^{-1}a_{2,4}a_{2,2}^{-1}. 
$$
For i=1,2, $\phi_i=\phi_{\G(D_i)}$ is given by:
$$
a_{1,j}\mapsto a_{1,j+6}, a_{2,j}\mapsto a_{2,j+4}, \text{ for }0\leq j\leq 9 
$$ 
Then by Theorem~\ref{thm:rep type 2},  $f_i^{-1}t_i f_i$ gives the action of $h_D$  on $\B(\p_{D_i})$.
Hence, we have\\

$l_i\mapsto\left\{ \begin{array}
{r@{\quad : \quad}l}
-m_1+l_2 & i=1\\ -m_1&  i=2\\ -2l_4+m_3+m_4 & i=3\\ l_3-l_4 & i=4
\end{array} \right.$\\\\

$m_i\mapsto\left\{ \begin{array}
{r@{\quad : \quad}l}
-l_2+m_2 & i=1\\l_1-2m_1+l_2-m_2&  i=2\\ -l_4 & i=3\\ l_3+m_3-l_4 & i=4
\end{array} \right.$\\

To understand the action of $h_D$ on $l_5$, we write $l_5=a_{2,0}\alpha$, where $\alpha$ is an arc on the handle with $1/5^{th}$ twist, joining end points of $a_{2,0}$, so that 
$$ a_{2,0}\alpha \to a_{2,4}\alpha \sim a_{2,4}a_{2,0}^{-1}a_{2,0}\alpha=-l_4+m_4+l_5,$$ under the action of $h_D$ on $H_1(S_5;\mathbb{Z}).$
Then by Theorem~\ref{thm:rep_roots}, $\Psi(h_D)$ is given by
\[\begin{bmatrix} \Psi(D_1) & ~ & ~  \\ ~ & \Psi(D_2) & B  \\ ~ & C &  I \\
\end{bmatrix},\]
 $$ \text{ where},\,\displaystyle B^t=\bigl(\begin{smallmatrix} 0 & 0 & 0 & 0\\ 0 & 0 & -1 & 1
\end{smallmatrix}\bigr) \text{ and } C=\bigl(\begin{smallmatrix} 0 & 0 & 0 & 0\\ -1 & -1 & -1 & 1
\end{smallmatrix}\bigr).$$
\end{example}

\subsection{Representations of roots of Dehn twists about multicurves}
For a multicurve $\mathcal{C}$ in $S_g$, conditions for the existence of a root $h$ of $t_{\mathcal{C}}$ of degree $n$ were derived in~\cite{KP}.  In general, such a root $h$ induces a nontrivial permutation of the curves in $\mathcal{C}$, and a finite order map $D_h$ on $S_g (C)$ (as in Theorem~\ref{thm:arb_real}). Moreover, the components of $S_g(\C)$ may themselves form orbits (called \textit{surface orbits}) under $D_h$. It is clear that the components of a single surface orbit must be homeomorphic to each other, so if a surface orbit has $m$ components homeomorphic to $S_g$, we denote it by $\S_g(m)$. Thus, a root of $t_{\C}$ induces a decomposition of $S_g(\C)$ in the form
$$
S_g(\C) = \sqcup_{i=1}^s \S_{g_i}(m_i).
$$
Furthermore, the restriction $D_i'$ of $h$ to each $\S_{g_i}(m_i)$ is the composition of a cyclic action $D_i$ on $S_{g_i}$ with a cyclical permutation of the components of $\S_{g_i}(m_i)$. The $D_i'$ must be pairwise \textit{compatible} across distinguished orbits, in the sense that the orbits must be of the same size, and the induced angles of rotation associated with the orbits must add up to allow the $D_i'$ to extend to a root. Using Theorem~\ref{thm:arb_real}, we can build (i.e. realize) each $D_i$ from finitely many pairwise compatibilities between irreducible finite order actions. Finally, by obtaining a suitable extension of Theorems~\ref{thm:rep_roots} and~\ref{thm:sep_curve_root_hom} to multicurves, one can obtain $\Psi(h)$. 

\section*{Acknowledgements}
The authors would like to thank Siddhartha Sarkar for helpful discussions on the symplectic representations of cyclic actions. 

\bibliographystyle{plain}      
\bibliography{geom_real}  
\end{document}